\documentclass[11pt,oneside]{amsart}
\usepackage{amsmath,amsfonts,amsgen,amstext,amsbsy,amsopn,amsthm, amssymb}
\usepackage{multirow, verbatim, cancel, mathabx}
\usepackage{bbold}
\usepackage[latin1]{inputenc}
\usepackage{enumitem}
\usepackage[T1]{fontenc}
\usepackage[colorlinks=true,hyperindex=true]{hyperref}
\usepackage{tikz}
\usetikzlibrary{calc}
\usepackage{verbatim}
\usetikzlibrary{arrows,chains,matrix,positioning,scopes}
\usepackage{mathrsfs}
\usepackage{color,soul}
\usepackage{xcolor}
\usepackage{tabulary}
\usepackage{pgfplots}
\usepackage{mathrsfs}
\usepackage{bbm}
\usepackage{dsfont}
\usepackage{bm}
\usepackage{mathtools}

\usepackage{color}
\newcommand\blue[1]{\textcolor{blue}{#1}}
\newcommand\red[1]{\textcolor{red}{#1}}

\usepackage{tabu} 
\usepackage{booktabs}
\newcommand\ChangeRT[1]{\noalign{\hrule height #1}}

\usepackage[mathscr]{euscript}

\DeclareSymbolFont{rsfs}{U}{rsfs}{m}{n}
\DeclareSymbolFontAlphabet{\mathscrsfs}{rsfs}

\makeatletter

\usepackage{tkz-fct} 
\usepackage{xcolor}

\usepackage{tkz-fct} 
\usepackage{pgfplots}
\usepackage{setspace}

\usetikzlibrary{matrix}
\usepackage{tikz}

\numberwithin{equation}{section}

\newcounter{smallarabics}

\newcounter{smallroman}
\newenvironment{romanenumerate}
{\begin{list}{{\normalfont\textrm{(\roman{smallroman})}}}
  {\usecounter{smallroman}\setlength{\itemindent}{0cm}
   \setlength{\leftmargin}{5ex}\setlength{\labelwidth}{4ex}
   \setlength{\topsep}{0.75\parsep}\setlength{\partopsep}{0ex}
   \setlength{\itemsep}{0ex}}}
{\end{list}}

\newcommand{\ben}{\begin{romanenumerate}}  
\newcommand{\een}{\end{romanenumerate}}  

\newtheorem{theoreme}{theorem }[section]
\newtheorem{theorem}[theoreme]{Theorem}
\newtheorem{proposition}[theoreme]{Proposition}
\newtheorem{Lemma}[theoreme]{Lemma}

\newtheorem{corollary}[theoreme]{Corollary}
\newtheorem{remark}{Remark}[section]
\newtheorem{example}[theoreme]{Example}

\newcommand{\Pp}{P^{\perp}}

\newcommand{\tH}{\theta(H)}

\newcommand{\eH}{\eta(H)}
\newcommand{\eD}{\eta(\Delta)}
\newcommand{\pp}{\frac{\partial \tilde{\varphi}}{\partial \overline{z}}(z)}

\newcommand{\A}{(z-A/R)^{-1}}

\newcommand{\JapA}{\Big \langle \frac{A}{R} \Big \rangle }

\newcommand{\const}{\frac{\i}{2\pi}}
\newcommand{\dz}{dz\wedge d\overline{z}}

\newcommand{\dN}{\bm{N}}
\newcommand{\dA}{\bm{A}}

\usepackage{array}
\newcolumntype{L}{>{\centering\arraybackslash}m{3cm}}

\newcommand\nn\nonumber
\renewcommand\leq\varleq
\renewcommand\geq\vargeq

\renewcommand{\proof}{\noindent \emph{Proof. }}
 
 \newcommand{\R}{\mathbb{R}}
 \newcommand{\N}{\mathbb{N}}
\newcommand{\Z}{\mathbb{Z}} \newcommand{\C}{\mathbb{C}}

\newcommand{\grad}{\nabla}

\renewcommand{\d}{\mathrm{d}}
\renewcommand{\i}{\mathrm{i}}

\renewcommand{\epsilon}{\varepsilon}

\usepackage{tikz-cd}

\usepackage{xcolor}
\usepackage{dcolumn}
\usepackage{tabu}

\newcolumntype{A}{D{.}{.}{2.3}}

\usepackage{tikz,tkz-tab}
\usetikzlibrary{positioning}
\makeatletter
      \def\@setcopyright{}
      \def\serieslogo@{}
      \makeatother
\begin{document}

\author{Sylvain Gol\'enia and Marc-Adrien Mandich}
   \address{Univ. Bordeaux, CNRS, Bordeaux INP, IMB, UMR 5251,  F-33400, Talence, France}
   \email{sylvain.golenia@math.u-bordeaux.fr}
      \address{Independent researcher, Jersey City, 07305, NJ, USA}
	\email{marcadrien.mandich@gmail.com}
   

   \title[LAP for discrete Schr\"odinger operator]{Limiting absorption principle for discrete Schr{\"o}dinger operators with a Wigner-von Neumann potential and a slowly decaying potential}

   \begin{abstract}
We consider discrete Schr\"odinger operators on $\Z^d$ for which the perturbation consists of the sum of a long range type potential and a Wigner-von Neumann type potential. Still working in a framework of weighted Mourre theory, we improve the limiting absorption principle (LAP) that was obtained in \cite{Ma1}. To our knowledge, this is a new result even in the one-dimensional case. The improvement is twofold. It weakens the assumptions on the long range potential and provides better LAP weights. Both upgrades include logarithmic terms. The proof relies on the fact that some particular functions, that contain logarithmic terms, are operator monotone. This fact is proved using Loewner's theorem and Nevanlinna functions. 
   \end{abstract}

%

 

\maketitle
\hypersetup{linkbordercolor=black}
\hypersetup{linkcolor=blue}
\hypersetup{citecolor=blue}
\hypersetup{urlcolor=blue}

\section{Introduction}

The limiting absorption principle (LAP) is an important resolvent estimate in the spectral and scattering theory of quantum mechanical Hamiltonians, in particular it implies that the part of the continuous spectrum where it holds is actually purely absolutely continuous. The LAP also provides boundary values for the resolvent that are useful for scattering theory. For Schr\"odinger operators on $\R^d$, the LAP was derived for a large class of short and long range potentials, see e.g.\ \cite{A}, \cite{PSS}, \cite{ABG}, as well as for decaying oscillatory potentials such as the Wigner-von Neumann potential, see e.g.\ \cite{DMR}, \cite{RT1}, \cite{RT2}. Schr\"odinger operators with Wigner-von Neumann potentials \cite{NW} are of interest because when appropriately calibrated they may produce eigenvalues embedded in the absolutely continuous spectrum of the Hamiltonian, and are linked with the phenomenon of resonance. The Wigner-von Neumann type potentials are still very much an active area of research. For the multi-dimensional case, see e.g.\ \cite{FS}, \cite{GJ2}, \cite{J}, \cite{JM}, \cite{Mar1}, \cite{Mar2}, \cite{Mar3} and \cite{Ma1}; for the one-dimensional case, see e.g.\ \cite{Li2}, \cite{L1}, \cite{L2}, \cite{L3}, \cite {Sim}, \cite{JS}, \cite{KN}, \cite{NS} and \cite{KS}. 

Regarding the LAP for Schr\"odinger operators, historically a lot of effort was put into lessening the decay assumptions on the short/long range perturbations. In this regard a breakthrough was made in 1981 by E.\ Mourre, see \cite{Mo1} and \cite{Mo2}. Very roughly speaking, assuming the short range perturbation satisfies $V_{\text{sr}}(x) = O(|x|^{-2})$ and the long range perturbation satisfies $x \cdot \grad V_{\text{lr}}(x) = O(|x|^{-1})$ as $|x| \to +\infty$, as well as several other technicalities, Mourre's theory implied the following LAP for the Schr\"odinger operator $H = -\Delta + V_{\text{sr}} + V_{\text{lr}}$ on $\R^d$ over some open real interval $I$: for every fixed $s>1/2$,
\begin{equation}
\label{general_LAP111}
\sup_{z \in I^{\pm}} \| \langle A \rangle ^{-s} (H-z)^{-1} \langle A \rangle ^{-s} \| < \infty,  \quad I^{\pm} := \{ z \in \C : \mathrm{Re}(z) \in I, \ \pm \mathrm{Im}(z) >0 \}.
\end{equation}
Here $\langle x \rangle := \sqrt{1+|x|^2}$ and $A$ is some self-adjoint operator. The theory was improved and by the end of the $20^{\text{th}}$ century the book \cite{ABG} presented a more sophisticated but optimal abstract framework to treat the short/long range potentials. This framework now goes by the name of \textit{classical Mourre theory}. The Wigner-von Neumann potential is not in the scope of this framework because its decay is at best $O(|x|^{-1})$, see e.g.\ \cite[Proposition 5.4]{GJ2} for the continuous case and \cite[Proposition 4.2]{Ma1} for the discrete case.

The beginning of the $21^{\text{st}}$ century sees the emergence of new approaches to Mourre's commutator method with \cite{G} and \cite{GJ1}, and the so-called \textit{weighted Mourre theory}, see \cite{GJ2}, or the \textit{local Putnam-Lavine theory}, see \cite{J}. The classical theory revolves around the \textit{Mourre estimate}
\begin{equation}
\label{mourreEstimate}
E_I(H) [H, \i A] _{\circ} E_{I} (H) \geqslant  \gamma E_{I} (H) + K,
\end{equation}
where $E_I(H)$ is the spectral projection of $H$ onto the interval $I$, $[H, \i A]_{\circ}$ is the realization of the formal commutator between $H$ and $\i A$, $K$ is a compact operator, and $\gamma >0$. On the other hand the more recent theories are centered around an estimate of the type
\begin{equation}
\label{WeightedmourreEstimate11}
E_I(H) [H, \i \varphi(A)] _{\circ} E_{I} (H) \geqslant E_{I} (H) \langle A \rangle ^{-s} \left( \gamma  +K \right)  \langle A \rangle ^{-s}  E_{I} (H), \quad s>1/2,
\end{equation}
where $\varphi$ is the real-valued function 
\begin{equation}
\label{varphi}
\varphi(t) := \int_{-\infty} ^t \langle x \rangle ^{-2s} dx, \quad t \in \R.
\end{equation}
Key underlying ideas behind this choice are that 1) $\varphi$ is bounded for $s>1/2$ and, 2) $\varphi'(A) = \langle A \rangle ^{-2s}$, i.e.\ the derivative of $\varphi$ yields the appropriate weights for the LAP as in \eqref{general_LAP111}.

In \cite{GJ2} ideas of weighted Mourre theory are used to derive the LAP for the continuous Schr\"odinger operator $H = -\Delta + V_{\text{sr}} + V_{\text{lr}} + W$ on $\R^d$, where $W(x) = w \sin ( k |x|) / |x|$, $w,k \in \R$, is the Wigner-von Neumann potential, and the short/long range potentials respectively satisfy $V_{\text{sr}}(x) = O(|x|^{-1-\epsilon})$ and $x \cdot \grad V_{\text{lr}}(x) = O(|x|^{-\epsilon})$ as $|x| \to +\infty$ for some $\epsilon > 0$. The cost of including the $W$ was that the decay assumptions for $V_{\text{sr}}$ and $V_{\text{lr}}$ are suboptimal as per \cite[Chapter 7]{ABG}. Article \cite{Ma1} was an application of the ideas of \cite{GJ2} to corresponding (albeit non-radial) discrete Schr\"odinger operators on $\Z^d$, see Theorem \ref{lapy} below for a statement. 

This article is a sequel to \cite{Ma1}. The aim is to follow the weighted Mourre theory as in \cite{GJ2} and \cite{Ma1}, but with a new bounded function $\varphi$ in \eqref{WeightedmourreEstimate11} which involves logarithmic terms. To write down this function we need some notation. Denote $\ln(\cdot)$ the Napierian logarithm. Let $\ln_0 (x) := 1$, $\ln_1(x) := \ln(1+ x)$, and for integer $k \geqslant 2$, $\ln_k(x) := \ln \left(1+\ln_{k-1}(x)\right)$. Thus $k$ is the number of times the function $\ln(1+x)$ is composed with itself. For simplicity we choose the domains $\ln_k : [1,+\infty) \to +\infty$. Also denote $\ln_k ^p (x) := (\ln_k  (x) )^p$, $p \in \R$. We choose $\varphi$ to be
\begin{equation}
\label{new_varphi22}
\varphi(t) := \int_{-\infty} ^t \langle x \rangle ^{-1} \ln_{m+1} ^{-2p} (\langle x \rangle) \prod _{k=0} ^m \ln_k ^{-1}(\langle x \rangle) dx , \quad t \in \R,
\end{equation}
for some $m \in \N$ and $p > 1/2$. Note that $\varphi$ as in \eqref{new_varphi22} is an increasing function that is asymptotically equal to a constant minus a $O( \ln_{m+1} ^{1-2p} (t) )$ term for $t \to +\infty$, hence a bounded function.

More notation is needed to present our results. The position space is the lattice $\Z^d$ for integer $d \geqslant 1$. For $n = (n_1,...,n_d) \in \Z^d$, set $|n|^2 := n_1^2 + ...+n_d^2$. Consider the Hilbert space $\mathscr{H} := \ell^2(\Z^d)$. Let $S_i$ be the shift operator to the right on the $i^{\text{th}}$ coordinate :
\begin{equation*}
(S_i \psi)(n) := \psi(n_1,...,n_i-1,...,n_d), \quad \forall \ n \in \Z^d, \ \psi = (\psi(n))_{n \in \Z^d} \ \in \mathscr{H}, \ \text{and} \ 1 \leqslant i \leqslant d.
\label{ShiftOperator}
\end{equation*}  
The shifts to the left on the $i^{\text{th}}$ coordinate are $S_i^* = S_i ^{-1}$. The discrete Schr{\"o}dinger operator 
\begin{equation}
\label{Hamiltonian}
H := \Delta + V + W
\end{equation}
acts on $\mathscr{H}$, where $\Delta$ is the discrete Laplacian operator defined by
\begin{equation}
\Delta : = \sum _{i=1}^d 2 - S_i - S_i^* = \sum _{i=1}^d \Delta_i,
\label{DiscreteLaplacian}
\end{equation} 
$W$ is a Wigner-von Neumann potential, parametrized by $w \in \R$, $k \in (0,2\pi) \setminus \{ \pi\}$ and defined by
\begin{equation}
\label{Wiggy}
(W \psi)(n) := w \cdot \sin(k(n_1+...+n_d)) |n|^{-1} \psi(n), \quad \text{for all} \ n \in \Z^d, n\neq 0, \ \text{and} \ \psi \in \mathscr{H},
\end{equation}
$(W\psi)(0) := 0$, and $V$ is a multiplication operator by a bounded real-valued sequence $(V(n))_{n \in \Z^d}$ such that $(V  \psi)(n) := V(n)\psi(n)$. We will also consider the oscillating potential $W_2$ given by 
\begin{equation}
\label{Osc_wiggy}
(W_2 \psi )(n) := w \cdot (-1)^{n_1+...+n_d} \cdot |n|^{-1} \psi(n), \quad \text{for all} \ n \in \Z^d, n\neq 0, \quad \text{and} \quad (W_2\psi)(0) := 0.
\end{equation}
In this case let $H_2 := \Delta + V + W_2$. Denote $\tau_i V$, $\tau_i^* V$ the operators of multiplication acting by 
\begin{equation*}
(\tau_i V) \psi(n) := V(n_1,...,n_i-1,...n_d)\psi(n), \quad (\tau_i^* V) \psi(n) := V(n_1,...,n_i+1,...n_d)\psi(n).
\end{equation*} 

\noindent \textbf{Hypothesis (H) : } Assume there are $m \in \N$ and $0\leqslant r < q$ such that :
\begin{align}
\label{short_range33}
V(n) &= O \left(\ln_{m+1} ^{-q} (|n|) \prod_{k=0} ^m \ln_{k} ^{-r} (|n|) \right), \quad as \ |n| \to +\infty, \quad \text{and} \\
\label{long_range33}
\ n_i (V - \tau_iV)(n)  &= O \left(\ln_{m+1} ^{-q} (|n|) \prod_{k=0} ^m \ln_{k} ^{-r} (|n|) \right), \quad as  \ |n| \to +\infty, \quad \forall \ 1 \leqslant i \leqslant d.
\end{align}


\begin{remark}
If hypothesis \textbf{(H)} holds for some $m \in \N$ and $0 \leqslant r <q$, then it holds for any integer $m' \geqslant m$ and $0 \leqslant r' < q'$, with $r' \leqslant r$ and $q' \leqslant q$. 
\end{remark}
Under hypothesis \textbf{(H)}, $V(n) = o(1)$ as $|n| \to +\infty$. Thus $V+W$ is a compact perturbation of $\Delta$ and so $\sigma_{\text{ess}}(H) = \sigma (\Delta)$. A Fourier transformation calculation shows that the spectrum of $\Delta$ is purely absolutely continuous and $\sigma (\Delta) = [0,4d]$. The formulation of the LAP requires a conjugate self-adjoint operator. Let $N_i$ be the position operator on the $i^{\text{th}}$ coordinate defined by 
$$
(N_i \psi)(n) := n_i \psi(n), \quad \mathrm{Dom}[N_i] := \bigg\{ \psi \in \ell^2(\Z^d) : \sum_{n \in \Z^d} |n_i \psi(n)|^2 < \infty  \bigg\}.
$$
The conjugate operator to $H$ is the generator of dilations denoted $A$ and it is the closure of
\begin{equation}
\label{generatorDilations}
A_0 := \i\sum_{i=1}^d  2^{-1}(S_i^*+S_i) - (S_i^*-S_i)N_i  = \frac{\i}{2}  \sum_{i=1}^d  (S_i-S_i^*)N_i + N_i(S_i-S_i^*) 
\end{equation} 
with domain 
the compactly supported sequences. $A$ is self-adjoint, see e.g.\ \cite{GGo}. Let 
\begin{equation}
\begin{aligned}
\mu(H) &:= (0,4)\setminus \{E_{\pm}(k)\} \ \ \text{for} \ \ d=1, \\
\mu(H) &:= (0,E(k)) \cup (4d-E(k),4d) \ \ \text{for} \ \ d\geqslant 2. 
\label{MuH}
\end{aligned}
\end{equation}
Here $k$ is the angular frequency of $W$, $E_{\pm}(k) := 2 \pm 2\cos\left(k/2\right)$, $E(k) := 2E_-(k)$ for $k \in (0,\pi)$ and $E(k) := 2E_+(k)$ for $k \in (\pi,2\pi)$. The sets $\mu(H)$ are real numbers where a classical Mourre estimate \eqref{mourreEstimate} is known to hold for $H$ -- assuming the presence of the oscillatory perturbation, i.e.\ $W \neq 0$. The first basic result for $H$ is about the local finiteness of the point spectrum.
\begin{proposition} 
\label{prop111}
Let $d \geqslant 1$, and $V$ satisfy hypothesis \textbf{(H)} for some $m \in \N$ and $0 \leqslant r < q$. If $E \in \mu(H)$, then there is an open interval $I$ of $E$ that contains at most finitely many eigenvalues of $H$ (including multiplicities). The corresponding eigenfunctions, if any, decay sub-exponentially. Also, in dimension $1$ the point spectrum of $H$ is empty in $I$ whenever $I \subset \left(E_-(k), E_+(k) \right)$.
\end{proposition}
  
We have a similar result for $H_2$. Set $\mu(H_2) := [0,4d] \setminus \left( \{0,4,...,4d-4, 4d\} \cup \{2d\} \right)$.

\begin{proposition} 
\label{prop111_H2}
Let $d \geqslant 1$ and $V$ satisfy hypothesis \textbf{(H)} for some $m \in \N$ and $0 \leqslant r < q$. If $E \in \mu(H_2)$, then there is an open interval $I$ of $E$ that contains at most finitely many eigenvalues of $H_2$ (including multiplicities). The corresponding eigenfunctions, if any, decay sub-exponentially. In dimension $1$ the point spectrum of $H_2$ is void in $(0,4) \setminus \{2\}$. 
\end{proposition}

Denote $P^{\perp} := 1-P$, where $P$ is the projection onto the pure point spectral subspace of $H$. We introduce the LAP weights with logarithmic terms :
\begin{equation}
\label{weightsW2}
\mathscr{W}_{M+1} ^{-p} (x) := \langle x \rangle ^{-\frac{1}{2}} w_{M} ^{-p,-\frac{1}{2}} (x), \quad 
w_{M} ^{\alpha,\beta} (x) := \ln_{M+1} ^{\alpha} \left( \langle x \rangle \right) \prod_{k=0}  ^M \ln_{k} ^{\beta} \left( \langle x \rangle \right), \quad \alpha, \beta \in \R.
\end{equation}
Let $J^{\pm} := \{ z \in \C : \mathrm{Re}(z) \in J, \ \pm \mathrm{Im}(z) >0 \}$. The main result of this article is :

\begin{theorem}
\label{lapy3}
Let $d \geqslant 1$ and $V$ satisfy hypothesis \textbf{(H)} for some $m \in \N$ and $2=r <q$. If $E \in \mu(H)$, then there is an open interval $I$ of $E$ such that for any compact interval $J \subset I$, any integer $M \geqslant m$, and any $p>1/2$,
\begin{equation}
\label{LAP_this_article2}
\sup \limits_{z \in J^{\pm}} \Big \| \mathscr{W}_{M+1} ^{-p} \left( A \right) (H-z)^{-1}P^{\perp}  \mathscr{W}_{M+1} ^{-p}   \left(  A  \right) \Big \| < \infty.
\end{equation}
The following local decay estimate also holds for all $\psi \in \mathscr{H}$, $M \geqslant m$ and $p>1/2$ :
\begin{equation}
\label{local_decay2}
\int _{\R} \Big \| \mathscr{W}_{M+1} ^{-p} \left(  A  \right)  e^{-\i tH} P^{\perp} E_{J} (H) \psi \Big \|^2 dt < \infty.
\end{equation}
By Lemma \ref{weightsAtoN_general} the above estimates also hold for $\langle N \rangle  = (1+ |N_1|^2 + ... + |N_d|^2)^{1/2}$ instead of $\langle A \rangle$. Finally, the spectrum of $H$ is purely absolutely continuous on $J$ whenever $P =0$ on $J$. 
\end{theorem}

With the obvious substitutions, the statement and conclusion of Theorem \ref{lapy3} also hold for the Hamiltonian $H_2 = \Delta + V + W_2$.

\begin{remark}
If $W=0$ or $W_2=0$, then the conclusion of Theorem \ref{lapy3}, i.e.\ the existence of a neighborhood where the LAP holds, is valid for a larger set of energies, precisely $E \in [0,4d] \setminus \{0,4,...,4d-4, 4d\}$. 
\end{remark}

\begin{example}
Suppose $V(n) = O \left( |n|^{-1} \ln^{-q} |n| \right)$ as $|n| \to +\infty$ for some $q>2$. Then hypothesis \textbf{(H)} holds with $m=0$ and $r=2$. The LAP in Theorem \ref{lapy3} holds with the weights $\mathscr{W}_{M+1} ^{-p} \left(  A  \right)$ for any $M \geqslant 0$ and $p>1/2$, so in particular for $\mathscr{W}_{1} ^{-p} \left(  A  \right) = \langle A \rangle ^{-\frac{1}{2}} \ln^{-p} (1+\langle A \rangle)$, $p >1/2$.
\end{example}

\begin{example}
An interesting $1$-dimensional example is $H_2 = \Delta + V + W_2$ where $V(n) = 2(-1)^n \cdot (n \ln(n))^{-1}$ and $W_2(n) = (-1)^{n} \cdot n^{-1}$, $n \in \N$. It is due to C.\ Remling \cite{R}. Although the example is on the half-line the arguments presented here apply with minor considerations. So in particular the spectrum of this $H_2$ is purely a.c.\ on $(-2,2) \setminus \{0\}$ (in the notation of \cite{R}). 
\end{example}

For comparison, the statement of the principal result of \cite{Ma1} is added below, taking into account the technical improvement in \cite[Theorem 1.8]{Ma2} : 

\begin{theorem}
\label{lapy}
\cite{Ma1} Let $d \geqslant 1$ and $V$ satisfy $V(n) = O(|n|^{-\epsilon})$ and $n_i (V-\tau_i V)(n) = O(|n|^{-\epsilon})$ for some $\epsilon>0$, and $\forall 1 \leqslant i \leqslant d$. If $E \in \mu(H)$, then there exists an open interval $I$ containing $E$ such that for any compact interval $J \subset I$ and any $s>1/2$, $\sup \limits_{z \in J^{\pm}} \|\langle A \rangle ^{-s} (H-z)^{-1}P^{\perp} \langle A \rangle ^{-s} \| < \infty$.

\end{theorem}
Thus, Theorem \ref{lapy3} improves Theorem \ref{lapy} by weakening the decay assumptions on $V$, as well as improving the LAP weights. The article \cite{Ma1} considers a second Wigner-von Neumann potential, namely $W'(n) = \prod_{i=1} ^d \sin(k_i n_i) /n_i$. Although $W'$ could have been included in this article and similar results would have been derived, we did not do so to keep the size of the article reasonable. Another obvious comment is that we do not know how to effectively use commutator methods to study spectral properties of the discrete Hamiltonian $H$ (with $W \neq 0$) on $[0,4d] \setminus \mu(H)$, see the comments that follow \cite[Proposition 4.5]{Ma1}. For the continuous Schr\"odinger operators on $\R^d$, with $W \neq 0$, a similar limitation exists, see \cite{GJ2} and \cite{JM}, in that the LAP was established only on $(0, k^2/4)$ for $d \geqslant 2$. But interestingly, an argument was recently found to justify a LAP on $(0,+\infty) \setminus \{ k^2/4\}$ when the Wigner-von Neumann potential is radial, i.e.\ $W(x) = q\sin( k|x|) / |x|$, see \cite{J} and \cite[Section 3.5]{Mb}.

Fix $d=1$ to discuss Theorem \ref{lapy3} in relation to other one-dimensional results in the literature. 
In \cite{Sim}, the perturbation consists of a Wigner-von Neumann potential plus a $V \in \ell^1(\Z)$, and it is proved that the spectrum of $H$ is purely absolutely continuous on $(0,4) \setminus \{ E_{\pm}(k) \}$. Note that our assumption \eqref{long_range33}, with $2=r<q$, implies $(V-\tau V) \in \ell^1(\Z)$ but not necessarily $V \in \ell^1(\Z)$, (although the weaker assumption $1 = r < q$ is sufficient to have $(V-\tau V) \in \ell^1(\Z)$). Another recent result is \cite{Li1}, where it is shown that if the perturbation is $O(|n|^{-1})$, then $H$ does not have singular continuous spectrum. Although this criteria applies to the Wigner-von Neumann potential, it does not apply to all potentials $V$ in the scope of this article, since for example hypothesis \textbf{(H)} allows for $V(n) = \ln^{-q}(1+\langle n \rangle)$, $q>2$. Now suppose the absence of oscillations, i.e.\ $W = 0$ and $W_2 = 0$. Because $(V-\tau V) \in \ell^1(\Z)$, $V$ is of bounded variation, and so the spectra of $H$ and $H_2$ are purely a.c.\ on $(0,4)$, by a result due to B.\ Simon in \cite{Si}. 

Consider now $d \geqslant 1$ to discuss Theorem \ref{lapy3} in relation to the framework of Mourre theory as exposed in \cite{ABG}. First, it is important to cite \cite{BSa}, where the details of this framework are worked out for the multi-dimensional discrete Schr\"odinger operators on $\Z^d$. From this perspective, this article is an attempt to bridge the gap between \cite{Ma1} and \cite{BSa}. 

Recall the regularity classes $\mathcal{C}^{1,1}(A) \subset \mathcal{C}^{1}_u(A) \subset \mathcal{C}^{1} (A)$ that structure the classical Mourre theory, see Section \ref{OpReg}. In \cite[Chapter 7]{ABG}, it is shown that $\mathcal{C}^{1,1}(A)$ is an optimal class in that if the Hamiltonian $H \in \mathcal{C}^{1,1}(A)$ then a LAP holds for $H$, whereas there is an example where $H \in \mathcal{C}^{1}_u(A)$ but $H \not \in \mathcal{C}^{1,1}(A)$ and no LAP holds. In \cite[Proposition 4.2]{Ma1} it is proved that the Wigner von-Neumann potential $W$, and thereby $H$, are not even of class $\mathcal{C}^{1}_u(A)$. An analogous argument also gives $W_2, H_2 \notin \mathcal{C}_u^1(A)$, provided $w \neq 0$. Thus our Hamiltonian does not fall under the framework of classical Mourre theory. 

Let us discuss the assumptions on the short/long range perturbations $V_{\text{sr}}$, $V_{\text{lr}}$. The criteria in the classical theory (see \cite[Theorem 2.1]{BSa}, also \cite[Theorem 7.6.8]{ABG}) are :
\begin{equation}
\label{Sahbani_criterionSR}
\int_{1}^{\infty} \sup \limits_{\kappa < |n| < 2\kappa} |V_{\text{sr}}(n)| d\kappa < \infty, \quad \forall \ 1 \leqslant i \leqslant d,
\end{equation}
and
\begin{equation}
\label{Sahbani_criterionLR}
V_{\text{lr}}(n) = o(1) \ as \ |n| \to +\infty \quad \text{and} \quad \int_{1}^{\infty} \sup \limits_{\kappa < |n| < 2\kappa} |(V_{\text{lr}}-\tau_i V_{\text{lr}})(n)| d\kappa < \infty, \quad \forall \ 1 \leqslant i \leqslant d.
\end{equation}
Table \ref{tab:2} contrasts these criteria with our hypothesis in Theorem \ref{lapy3}. In this Table $g(n) := O(|n|^{-1} w_m^{-q,-r}(n))$ as $|n| \to +\infty$.

\begin{table}[ht]
\centering
\begin{tabular}{!{\vrule width 1.0pt}c|c|c|c!{\vrule width 1.0pt}}
\ChangeRT{1.0pt}
\footnotesize If $V$ satsifies & $V(n) = O(g(n))$  &   $(V-\tau_i V)(n) = O(g(n))$, $\forall i$  & $V(n) = \sum_{i=1}^d \frac{\ln( 2+|n_i|)}{\ln ^{\sigma}(1+\langle n \rangle)}$  \\ [0.3ex]
\ChangeRT{1.0pt} 
\footnotesize then : \eqref{short_range33} holds & $ yes, \forall m \in \N$ & $no$ & $yes, \forall \sigma >3$  \\ [0.3ex]
\footnotesize with $2=r < q$ &  &   &   \\ [0.3ex]  \hline
\footnotesize then : \eqref{long_range33} holds & $yes, \forall m \in \N$ & $yes, \forall m \in \N$  & $yes, \forall \sigma>2$  \\ [0.3ex]
\footnotesize with $2=r < q$ &  &   &   \\ [0.3ex]  \hline
\footnotesize then : \eqref{Sahbani_criterionSR} holds & $yes, \forall m \in \N$ & $no$ & $no, \forall d \geq 1$ \\  [0.3em]
\footnotesize & $and  \ 1=r<q$ &   &   \\ [0.3ex]  \hline
\footnotesize then : \eqref{Sahbani_criterionLR} holds  & $yes, \forall m \in \N$ & $yes, \forall m \in \N$ & $yes, for \  d=1$, $\sigma >1$;  \\  [0.3ex]  
\footnotesize & $and  \ 1=r<q$ & $and  \ 1=r<q$  & $no, \forall d \geq 2$  \\ [0.3ex]  \hline
\ChangeRT{1.0pt}
\end{tabular}
\caption{Examples.} 
\label{tab:2}
\end{table}

The last column in Table \ref{tab:2} shows that it is possible to concoct a non-radial potential $V$ that verifies the hypothesis of Theorem \ref{lapy3} but neither \eqref{Sahbani_criterionSR} nor \eqref{Sahbani_criterionLR}. It is an open problem for us to verify if this potential is of class $\mathcal{C}^{1,1}(A)$. If the potential $V$ is radial the first two lines of Table \ref{tab:2} suggest that our assumptions on $V$ may be slightly suboptimal, at least in the absence of oscillations, i.e.\ $W=0$, because Theorem \ref{lapy3} requires $r=2$, rather than $r=1$. To the best of our understanding, our requirement of "an entire extra logarithm" ($r=1$ vs.\ $r=2$) is due to a technical limitation in the construction of the almost analytical extension of the function $\varphi$, see \eqref{new_varphi22}, which is employed to apply the Helffer-Sj\"ostrand formula and associated functional calculus. Indeed, although $\varphi'(t) = \langle t \rangle^{-1} w_m ^{-2p,-1}(t)$ the extension only verifies $ | \partial \tilde{\varphi} / \partial \overline{z} | \leqslant c \langle \text{Re} (z) \rangle ^{-1-\ell} | \text{Im}(z) | ^{\ell}, \ell \in \N$. Clearly all logarithmic decay is lost in the process of extending to the complex plane.

Let us briefly discuss the LAP weights. Let $T$ be a self-adjoint operator on $\mathscr{H}$, $E_\Sigma (T)$ its spectral projection on a set $\Sigma \subset \R$. Let $s, p, p' \in \R$, $M \in \N$. Let $\Sigma_j = \{x \in \R : 2^{j-1} \leqslant |x| \leqslant 2^j \}$, $j \geqslant 1$, and $\Sigma_0 := \{ x\in \R : |x| \leqslant 1\}$. Define the Banach spaces  

\begin{equation*}
L^2_{s, p, p', M}(T) := \Big \{ \psi \in \mathscr{H} : \| \langle T \rangle ^{s} w_M^{p,p'}(T)  \psi \| < \infty \Big \}, \quad \text{and}
\end{equation*}
\begin{equation*}
B (T) := \Big \{ \psi \in \mathscr{H}  : \sum_{j=0} ^{\infty} \sqrt{2^j} \| E_{\Sigma_j} (T) \psi \|  < \infty \Big \}.
\end{equation*}
The norms on these spaces are respectively
$$ \| \psi \|_{L^2_{s,p,p',M} (T)} := \| \langle T \rangle ^{s} w_M^{p,p'}(T)  \psi \| \quad \text{and} \quad \| \psi \|_{B(T)} := \sum_{j=0} ^{\infty} \sqrt{2^j} \| E_{\Sigma_j} (T) \psi \|.$$
The dual of $L^2_{s, p, p', M}(T)$ with respect to the inner product on $\mathscr{H}$ is $( L^2_{s, p, p', M}(T) )^* = L^2_{-s, -p, -p', M}(T)$ and 
the dual $B^*(T)$ of $B(T)$ is the Banach space obtained by completing $\mathscr{H}$ in the norm 
$$ \| \psi \|_{B^*(T)} := \sup_{j \in \N} \sqrt{2^{-j}} \| E_{\Sigma_j} (T) \psi \|.$$
We refer to \cite{JP} and the references therein for these definitions. Write $L^2_{s}(T) := L^2_{s, 0, 0, 0}(T)$. For any $s, p>1/2$ and $M \in \N$, the following strict inclusions hold : 
$$L^2_s(T) \subsetneq L^2_{1/2, p,1/2, M}(T) \subsetneq B(T)  \subsetneq L^2_{1/2}(T),$$
and 
$$L^2_{-1/2}(T) \subsetneq B^*(T)  \subsetneq L^2_{-1/2, -p,-1/2, M}(T) \subsetneq L^2_{-s}(T).$$
Also clear is that $L^2_{1/2, p,1/2, M}(T) \subset L^2_{1/2, p',1/2, M'}(T)$ whenever $p' \leqslant p$ and $M' \geqslant M$.

Let $\mathscr{E}(H)$ be the set of eigenvalues and thresholds of $H$. By threshold we mean a real number $E$ for which no Mourre estimate holds for $H$ wrt.\ $A$, regardless of the interval $I$, $I \ni E$. 
Classical Mourre theory says that under appropriate conditions for $V$, and $W=0$, then the LAP
\begin{equation}
\label{optimal}
\sup \limits_{\eta > 0} \| (H-\lambda - \i \eta)^{-1} \psi \|_{\mathcal{K}^*} \leqslant c(\lambda) \| \psi \|_{\mathcal{K}}
\end{equation}
holds for some appropriate pair of Banach spaces $(\mathcal{K}, \mathcal{K}^*)$, some $c(\lambda) > 0$ and all $\lambda \in \R \setminus \mathscr{E}(H)$. Also, $c(\lambda)$ can be chosen uniform in $\lambda$ over fixed compact subsets of $\R \setminus \mathscr{E}(H)$. On the one hand, the spaces $(\mathcal{K}, \mathcal{K}^*) = (B(A),B^*(A))$ are optimal in a certain sense, see \cite{JP}, \cite{AH} and the discussion at the beginning of \cite[Chapter 7]{ABG}, but these are not Besov spaces. On the other hand, the Besov spaces $(\mathcal{K}, \mathcal{K}^*) = (\mathcal{H}_{1/2,1}, \mathcal{H}_{-1/2,\infty})$ which appear in \cite{ABG} and \cite{BSa} are not as optimal as $(B(A),B^*(A))$, but are optimal in the scale of Besov spaces and allow for a larger class of potentials $V$. Note that our new result Theorem \ref{lapy3} implies \eqref{optimal} for $\mathcal{K} = L^2_{1/2, p,1/2, M}(A)$, any $p >1/2$, $M \geqslant m$, while Theorem \ref{lapy} implies \eqref{optimal} for $\mathcal{K} = L^2_{s}(A)$, any $s >1/2$.

Finally, a few comments about the proof of Theorem \ref{lapy3}. In essence it is the same as in \cite{GJ2}, or \cite{Ma1}. The difference lies in the function $\varphi$ given by \eqref{new_varphi22} rather than \eqref{varphi}. 
But in order to formulate meaningful conditions on $V$ (hypothesis \textbf{(H)}) this translates into a problem of bounding functions of the generator of dilations $A$ by functions of the operator of position $|N|$ involving logarithmic terms. More specifically, while the proof of Theorem \ref{lapy} required only $\langle A \rangle ^{\epsilon} \langle N \rangle ^{-\epsilon} \in \mathscr{B}(\mathscr{H})$, the space of bounded operators on $\mathscr{H}$, the proof of Theorem \ref{lapy3} requires 
$w_{M} ^{\alpha,\beta} (A) \cdot w_{M} ^{-\alpha,-\beta} (N) \in \mathscr{B}(\mathscr{H})$ for some appropriate $\alpha, \beta \geqslant 0$, which is equivalent to $w_{M} ^{2\alpha,2\beta} (A) \leqslant w_{M} ^{-2\alpha,-2\beta} (N)$
in the sense of forms. To achieve this we make a detour via the polylogarithm functions $\mathrm{Li}_{\sigma}(z)$ (specifically the ones of order $2 < \sigma \leqslant 3$ are used). The reason for doing this is that higher powers of the logarithm are not Nevanlinna functions, but the functions 
$\mathrm{Li}_{\sigma}(z)$, $\text{Re}(\sigma) >0$, are Nevanlinna functions which continue analytically across $(-1,+\infty)$ from $\C_+$ to $\C_-$, and are thus operator monotone functions, see Section \ref{Polylog}.

The plan of the article is as follows. In Section \ref{classicalMourreTheory}, we recall definitions and results of classical and weighted Mourre theory. In Section \ref{VerifyMourre} we derive the classical Mourre estimate for the Hamiltonians $H$ and $H_2$ and prove Propositions \ref{prop111} and \ref{prop111_H2}. In Section \ref{section4normbounds} we derive operator norm bounds involving logarithms of self-adjoint operators. In Section \ref{lemmaSection} we prove key Lemmas that are needed for the proof of Theorem \ref{lapy3}. In Section \ref{PROOF} we prove Theorem \ref{lapy3}. 
In appendix~A (Section~\ref{Nevanlinna}) we review Loewner's theorem which makes the connection between Nevanlinna functions and operator monotone functions, and prove an extension of Loewner's theorem that is suitable for semi-bounded self-adjoint operators. 
In appendix~B (Section \ref{Polylog}) we briefly review polylogarithms and explain that the ones of positive order are Nevanlinna functions. 
In appendix~C (Section \ref{Helffer-Appendix}) we mention the basic tools we use from the kit of almost analytic extensions and Helffer-Sj\"ostrand functional calculus.

\noindent{\textbf{Acknowledgements :}} It is a pleasure to thank Thierry Jecko for fruitful conversations on the topic, generous advice to improve the results in various ways, and especially giving us the permission to publish his Proposition \ref{Proposition ThierryJecko} which conveniently generalizes self-adjoint operator norm estimates to sub-multiplicative functions. We are grateful to the two anonymous referees for a careful reading of the manuscript leading to many improvements, and especially Proposition \ref{Proposition ThierryJecko}. 







\section{Basics of the abstract classical and weighted Mourre theories}
\label{classicalMourreTheory}

\subsection{Operator Regularity}
\label{OpReg}

We consider two self-adjoint operators $T$ and $A$ acting in some complex Hilbert space $\mathcal{H}$, and for the purpose of this brief overview $T$ will be bounded. Given $k \in \N$, we say that $T$ is of class $\mathcal{C}^k(A)$, and write $T \in \mathcal{C}^k(A)$, if the map  
\begin{equation}
\label{defCk}
\R \ni t \mapsto e^{\i tA}Te^{-\i tA} \in \mathscr{B}(\mathcal{H})
\end{equation} 
has the usual $C^k(\R)$ regularity with $\mathscr{B}(\mathcal{H})$ endowed with the strong operator topology. The form $[T,A]$ is defined on $\text{Dom}[A] \times\text{Dom}[A]$ by $\langle \psi , [T,A] \phi \rangle := \langle T\psi,A \phi \rangle - \langle A\psi , T\phi \rangle$. Recall the following convenient result:
\begin{proposition} 
\label{Prop629}
\cite[Lemma 6.2.9]{ABG} Let $T \in \mathscr{B}(\mathcal{H})$. The following are equivalent:
\begin{enumerate}
\item $T \in \mathcal{C}^1(A)$.
\item The form $[T,A]$ extends to a bounded form on $\mathcal{H}\times \mathcal{H}$ defining a bounded operator denoted by $\text{ad}^1_A(T) := [T,A]_{\circ}$.
\item $T$ preserves $\text{Dom}[A]$ and the operator $TA-AT$, defined on $\text{Dom}[A]$, extends to a bounded operator.
\end{enumerate}
\end{proposition}
Consequently, $T \in \mathcal{C}^k(A)$ if and only if the iterated commutators $\text{ad}^p_A(T) := [\text{ad}^{p-1}_A(T), A]_{\circ}$ are bounded for $1 \leqslant p \leqslant k$. We say that $T \in \mathcal{C}^{k}_u(A)$ if the map \eqref{defCk} has the $C^k(\R)$ regularity with $\mathscr{B}(\mathcal{H})$ endowed with the norm operator topology. We say that $T \in \mathcal{C}^{1,1}(A)$ if 
\begin{equation*}
\int_0 ^1 \| [T, e^{\i tA} ]_{\circ} , e^{\i tA}]_{\circ} \| t^{-2} dt < \infty.
\end{equation*}
It turns out that $\mathcal{C}^2(A) \subset \mathcal{C}^{1,1}(A) \subset \mathcal{C}^{1}_u (A) \subset \mathcal{C}^1(A)$.

\subsection{The Mourre estimate}
\label{Mestimate_section}
Let $I$ be an open interval and assume $T \in \mathcal{C}^1(A)$. We say that the \textit{Mourre estimate} holds for $T$ on $I$ if there is $\gamma > 0$ and a compact operator $K$ such that 
\begin{equation}
E_{I}(T) [T,\i A ]_{\circ} E_{I}(T)  \geqslant \gamma  E_{I}(T) +K
\label{Mourre K}
\end{equation}
in the form sense on $\text{Dom}[A] \times\text{Dom}[A]$. We say that the \textit{strict Mourre estimate} holds for $T$ on $I$ if \eqref{Mourre K} holds with $K=0$. Assuming the estimate holds over $I$, $T$ has at most finitely many eigenvalues in $I$, and they are of finite multiplicity, while if the strict estimate holds $T$ has no eigenvalues in $I$. This is a direct consequence of the Virial Theorem, see \cite[Proposition 7.2.10]{ABG}. Let $I(E;\epsilon)$ be the open interval of radius $\epsilon>0$ centered at $E \in \R$. One considers the function $\varrho_{T}^A :\R \mapsto \overline{\R}$: 
\begin{equation*}
\varrho_{T}^A : E  \mapsto \sup \ \big\{ \gamma \in \R: \exists \epsilon>0 \  \text{such that} \  E_{I(E;\epsilon)}(T)[T,\i A]_{\circ} E_{I(E;\epsilon)}(T) \geqslant \gamma \cdot E_{I(E;\epsilon)}(T) \big\}.
\label{VARRHO}
\end{equation*}
It is known for example that $\varrho_{T}^A$ is lower semicontinuous and $\varrho_{T}^A(E) < \infty$ if and only if $E \in \sigma(T)$. For more properties of this function, see \cite[chapter 7]{ABG}. 

\subsection{The weighted Mourre estimate \& LAP}
\label{weightedMestimate}

Let $T$ be self-adjoint in a Hilbert space $\mathcal{H}$. Denote $P^{\perp} := 1-P$, where $P$ is the projection onto the pure point spectral subspace of $T$.  In \cite{GJ2} it is explained in a more general setting that a \emph{projected weighted estimate} of the form
\begin{equation}
\label{general_weighted}
\Pp E_I(T) [T, \i B]_{\circ} E_I (T) \Pp \geqslant c \Pp E_I (T) C^2 E_I (T) \Pp,
\end{equation}
where $c > 0$, $B,C$ are any linear operators satisfying $C B C^{-1} \in \mathscr{B}(\mathcal{H})$, and $C$ self-adjoint and injective, will imply a LAP of the form
\begin{equation}
\label{general_LAP}
\sup \limits_{z \in I^{\pm}} \| C (T-z)^{-1} \Pp C \| < \infty.
\end{equation}
The proof is by contradiction and therefore does not provide a description of the continuity properties of the LAP resolvent. Nonetheless, the resolvent estimate \eqref{general_LAP} implies the absence of s.c.\ spectrum for $T$ in $I$, and if $B$ is $T$-bounded then \eqref{general_weighted} also gives the time decay estimate 
\begin{equation}
\label{time_Decay}
\int_{\R} \| C e^{\i t T }E_I (T) P^{\perp} \psi \|^2 dt < \infty, \quad \psi \in \mathcal{H}.
\end{equation}


\section{Verifying the classical Mourre estimate}
\label{VerifyMourre}
In this Section we derive the Mourre estimate described in the previous Section for $T=H$ given by \eqref{Hamiltonian}. But first we must discuss operator regularity. 
\subsection{Verifying the operator regularity} We use freely Proposition \ref{Prop629}. Let
\begin{equation}
\label{thresholdsDelta}
\boldsymbol{\Theta}(\Delta) := \{ \boldsymbol{t}_0, \boldsymbol{t}_1, ..., \boldsymbol{t}_d\}, \quad \text{where} \quad \boldsymbol{t}_k = 4k, k=0, 1, ...,d.
\end{equation}
Recall that our choice of conjugate operator $A$ is the closure of \eqref{generatorDilations}. This choice is justified by the fact that the commutator between $\Delta$ and $A$ is :
\begin{equation}
\label{commutator}
[\Delta,\i A]_{\circ} = \sum_{k=1}^d \Delta_k(4-\Delta_k).
\end{equation} 
In particular $\Delta \in \mathcal{C}^1(A)$, and since $\sigma(\Delta_k) = [0,4]$, $\forall k=1, ...,d$, the strict Mourre estimate holds for $\Delta$ and $A$ over any interval $I \Subset \sigma(\Delta) \setminus \boldsymbol{\Theta}(\Delta) = [0,4d] \setminus \boldsymbol{\Theta}(\Delta)$. For this reason, we refer to $\boldsymbol{\Theta}(\Delta)$ as the \textit{thresholds} of $\Delta$ (they are also thresholds for $H$ and $H_2$). In fact, one can be more specific and prove that if $E \in \sigma_k(\Delta) := [\boldsymbol{t}_{k-1}, \boldsymbol{t}_{k}]$, for some $k=1,...,d$, then

$$\varrho_{\Delta}^A(E) = -(E-\boldsymbol{t}_{k-1})(E-\boldsymbol{t}_{k}).$$
This can be proved by induction on $d$ and using the nifty result \cite[Theorem 8.3.6]{ABG}. One can further show that for all $k=1,...,d$, $[\Delta, \i A]_{\circ}$ decomposes into the sum of a non-negative operator and a non-negative remainder $b_k(\Delta)$, namely
\begin{equation*}
\label{b_k}
[\Delta, \i A]_{\circ} = -(\Delta-\boldsymbol{t}_{k-1})(\Delta-\boldsymbol{t}_{k}) + b_k(\Delta),  
\end{equation*}
where 
$$b_k(\Delta) := - 8(k-1)\Delta + 16k(k-1) + \sum_{\substack{1 \leqslant i,j \leqslant d \\ i \neq j}} \Delta_i \Delta_j.$$
A somewhat surprising direct calculation based on functional calculus may show that no strict positivity can be extracted from $b_k(\Delta)$ by localizing in energy. Finally, it is not hard to prove that $\Delta \in \mathcal{C}^2(A)$, or $\mathcal{C}^{\infty}(A)$ for that matter. Moving on to the commutator between the potential $V$ and $A$, we have:
\begin{equation}
\label{commutatorV}
[V,\i A]_{\circ} = \sum_{i=1}^d (2^{-1}-N_i)(V-\tau_iV)S_i + (2^{-1}+N_i)(V-\tau_i^*V)S_i^*.
\end{equation}
It is readily seen that the assumption \eqref{long_range33}, for some $m\in \N$, $0 \leqslant r < q$, implies that $[V, \i A]_{\circ}$ is a compact operator.  In particular $V \in \mathcal{C}^1(A)$. Regarding the commutator between the Wigner-von Neumann potential $W$ and $A$, we have $[W, \i A]_{\circ} = K_W + B_W$ where
\begin{align}
\label{KW}
K_W &:= 2^{-1} W \sum _{i=1} ^d (S_i^* +S_i) + 2^{-1} \sum _{i=1} ^d (S_i^* +S_i) W, \\
\label{BW}
B_W &:= \sum _{i=1}^d U_i \tilde{W} (S_i^*-S_i) - \sum _{i=1}^d (S_i^*-S_i) \tilde{W} U_i,
\end{align}
$\tilde{W}$ is the operator $(\tilde{W}\psi)(n) := w \cdot \sin(k(n_1+...+n_d))\psi(n)$ and $U_i$ is the operator $(U_i \psi)(n) := n_i |n|^{-1} \psi(n)$, for $n\neq 0$, $(U_i \psi)(0) := 0$. Note that $K_W$ is compact and $B_W$ is bounded. Thus $W \in \mathcal{C}^1(A)$. However, $W \notin \mathcal{C}_u^1(A)$, see \cite[Proposition 4.2]{Ma1}. Finally, regarding the commutator between the oscillating potential $W_2$ and $A$, we have $[W_2, \i A]_{\circ} = K_{W_2} + B_{W_2}$ where
\begin{align}
\label{KW22}
K_{W_2} &:= 2^{-1} W_2 \sum _{i=1} ^d (S_i^* +S_i) + 2^{-1} \sum _{i=1} ^d (S_i^* +S_i) W_2, \\
\label{BW22}
B_{W_2} &:= \sum _{i=1}^d U_i \sigma (S_i^*-S_i) - \sum _{i=1}^d (S_i^*-S_i) \sigma U_i,
\end{align}
$\sigma$ is the operator $(\sigma \psi)(n) := w \cdot (-1)^{n_1+...+n_d} \psi(n)$. Again, $K_{W_2}$ is compact and $B_{W_2}$ is bounded. Thus $W_2 \in \mathcal{C}^1(A)$. An analogous proof to that given in \cite[Proposition 4.2]{Ma1} and \cite[Proposition 3.3]{Ma1} reveals that $W_2 \notin \mathcal{C}_u^1(A)$. 

\subsection{Establishing the Mourre estimate for $H$ wrt.\ $A$}

The following Lemma applies to the one-dimensional Laplacian ($d=1$).
\begin{Lemma} \cite[Lemma 3.4]{Ma1}
\label{ImportantLemma}
Recall that $E_{\pm}(k) := 2 \pm 2\cos (k/2)$. Let $E \in [0,4]\setminus\{E_{\pm}(k) \}$. Then there exists $\epsilon = \epsilon(E) >0$ such that for all $\theta \in C_c^{\infty}(\R)$ supported on $I=(E-\epsilon, E+\epsilon)$, $\theta(\Delta)\tilde{W}\theta(\Delta)=0$. Thus $\theta(\Delta) B_{W} \theta(\Delta)$ is a compact operator.
\end{Lemma}

The following Lemma applies to the multi-dimensional Laplacian ($d\geqslant 2$).
\begin{proposition} \cite[Proposition 4.5]{Ma1}
\label{prop45}
Let 
\begin{equation}
\label{Lambdas}
    E(k) := \begin{cases}
    4-4\cos(k/2) & \ \text{for} \ k \in (0,\pi) \\
    4+4\cos(k/2) & \ \text{for} \ k \in (\pi,2\pi)
    \end{cases} \quad \text{and} \quad \mu(H) := (0,E(k)) \cup (4d-E(k),4d). 
\end{equation}
For each $E \in \mu(H)$ there exists $\epsilon = \epsilon(E) > 0$ such that for all $\theta \in C_c^{\infty}(\R)$ supported on $I := (E-\epsilon, E+\epsilon)$, $\theta(\Delta) \tilde{W} \theta(\Delta) = 0$. In particular, $\theta(\Delta) B_W \theta(\Delta)$ is a compact operator.
\end{proposition}

Putting together everything discussed in this Section we get the Mourre estimate for $H$: 

\begin{proposition} 
\label{Prop13}
Let $d\geqslant 1$, $H = \Delta+ V+W$ and $\mu(H)$ be as defined \eqref{MuH}. Suppose that $V$ satisfies \eqref{long_range33} for some $m \in \N$ and any $r,q \in [0,\infty)$, not both zero. Then $H \in \mathcal{C}^1(A)$, and for any $I \Subset \mu(H)$, there are $\gamma>0$ and compact $K$ such that the Mourre estimate 
$E_{I}(H) [H,\i A ]_{\circ} E_{I}(H)  \geqslant \gamma E_{I}(H) +K$ holds. In particular the conclusion of Proposition \ref{prop111} holds on the interval $I$.
\end{proposition}
Proposition \ref{Prop13} is basically a reformulation of Proposition \ref{prop111}. The proof is easy and goes along the lines of \cite[Proposition 3.5]{Ma1}. The decay of the eigenfunctions is a consequence of \cite[Theorem 1.5]{Ma2}, while the absence of eigenvalues is a consequence of \cite[Theorem 1.2]{Ma2}.  Note that if $H = \Delta + V$, with $V$ as in Proposition \ref{Prop13} and $W =0$, then the above shows that we have a Mourre estimate for $H$ on any $I \Subset [0,4d] \setminus \boldsymbol{\Theta}(\Delta)$.

\subsection{Establishing the Mourre estimate for $H_2$ wrt.\ $A$}
The following Lemma applies to all dimensions ($d\geqslant 1$).

\begin{Lemma} 
\label{Lemma_new_2}
For each $E \in [0,4d] \setminus \{2d\}$ there exists $\epsilon = \epsilon(E) > 0$ such that for all $\theta \in C_c^{\infty}(\R)$ supported on $I := (E-\epsilon, E+\epsilon)$, $\theta(\Delta) \sigma \theta(\Delta) = 0$. In particular, $\theta(\Delta) B_{W_2} \theta(\Delta)$ is a compact operator.
\end{Lemma}
\begin{proof}
The relation $\Delta \sigma = \sigma (4d -\Delta)$ holds. This entails $\theta(\Delta) \sigma \theta(\Delta) = \sigma \theta(\Delta)  \theta(4d-\Delta)$ for all $\theta \in C_c^{\infty}(\R)$. Also, $\theta(\Delta)  \theta(4d-\Delta) = 0$ whenever supp$(\theta) \Subset [0,2d)$ or $\Subset (2d, 4d]$. So it suffices to take $\epsilon$ sufficiently small. The compactness of $\theta(\Delta) B_{W_2} \theta(\Delta)$ is then a consequence, together with the fact that $[U_i, \Delta]$ is compact.
\qed
\end{proof}

Putting together everything discussed above we get the Mourre estimate for $H_2$: 

\begin{proposition} 
\label{Prop13_h2}
Let $d\geqslant 1$, $H_2 = \Delta+ V+W_2$ and $\mu(H_2) := [0,4d] \setminus \left( \boldsymbol{\Theta}(\Delta) \cup \{2d\} \right)$. Suppose that $V$ satisfies \eqref{long_range33} for some $m \in \N$ and any $r,q \in [0,\infty)$, not both zero. Then $H_2 \in \mathcal{C}^1(A)$, and for any $I \Subset \mu(H_2)$, there are $\gamma>0$ and compact $K$ such that the Mourre estimate 
$E_{I}(H_2) [H_2,\i A ]_{\circ} E_{I}(H_2)  \geqslant \gamma E_{I}(H_2) +K$ holds. In particular the conclusion of Proposition \ref{prop111_H2} holds on the interval $I$.
\end{proposition}
Proposition \ref{Prop13_h2} is basically a reformulation of Proposition \ref{prop111_H2}. All the comments done after Proposition \ref{Prop13} apply here to $H_2$ as well.











\section{Operator norm bounds for functions of self-adjoint operators}
\label{section4normbounds}

\subsection{Weighted resolvent bounds with logarithmic terms} Let $Q$ be a self-adjoint operator in a Hilbert space $\mathscr{H}$, $z \in \C$. The estimate $\| (Q-z)^{-1}\| \leqslant |\text{Im}(z)|^{-1}$ is widely known. The aim of this sub-section is to obtain bounds of the form $\| f(Q)(Q-z)^{-1}\| \leqslant f(\text{Re}(z)) |\text{Im}(z)|^{-1}$ for some specific classes of functions $f : \R^+\mapsto \R^+$ containing logarithmic terms. In what follows we state results for a general self-adjoint operator $Q$. The results of this sub-section are later applied in 2 ways : 1) in the next 2 sub-sections (which are still in a general framework), and 2) in the proof of Theorem \ref{lapy3} (our Schr\"odinger problem). In the latter application,  Lemma \ref{infamousLemma3} is applied with $Q=A$, the discrete version of the generator of dilations, $s=1/2$, and $p_k$ both positive and negative for $0 \leqslant k \leqslant m+1$.


The next Lemma has already been proved, e.g. \cite{DG, G, GJ2}. We propose a proof by contradiction which  lends itself reasonably well to other types of functions. Lemmas \ref{infamousLemma} and \ref{infamousLemma3}  extend Lemma \ref{infamousLemma1} by including logarithmic functions. Finally, Proposition \ref{Proposition ThierryJecko} is a very nice generalization of Lemmas \ref{infamousLemma} and \ref{infamousLemma3} to a wider class of sub-multiplicative functions. The proof is entirely due to Thierry Jecko and we are indebted to him for allowing us to publish it. 

\begin{Lemma}
\label{infamousLemma1} 
Let $Q$ be a self-adjoint operator. Let $\Omega := \{ (x,y) \in \R^2 : 0 < |y| \leqslant c \langle x \rangle \}$, for some $c>0$. Then for every $0 \leqslant s \leqslant 1$ there exists $C > 0$ such that for all $z = x+\i y \in \Omega$ :
\begin{equation}
\| \langle Q \rangle ^s (Q-z)^{-1} \| \leqslant C \langle x \rangle ^s  |y|^{-1}.
\label{norm1}
\end{equation}
\end{Lemma}
\proof 
If $s=0$ one can take $C=1$. Suppose that $s>0$. By the spectral theorem, 
\[ \| \langle Q \rangle ^s (Q-z)^{-1}  \| ^2 \leqslant \sup_{t \in \R} \  \langle t \rangle ^{2s} \left( (t-x)^2+y^2 \right)^{-1}.\]
Let
\begin{equation*}
f_s(x,y,t) : =  \frac{\langle t \rangle ^{2s}}{\langle x \rangle ^{2s}} \frac{y^2}{(t-x)^2+y^2}, \quad \text{defined on} \ (x,y,t) \in \overline{\Omega} \times \R \setminus \Lambda,
\end{equation*}
where $\overline{\Omega} = \{ (x,y) \in \R^2 : |y| \leqslant c \langle x \rangle \}$ and $\Lambda := \{ (x,y,t) \in \overline{\Omega} \times \R : y=0, x=t\}$. 
To prove the Lemma, it is enough to show that $f_s$ is uniformly bounded on the entire region $\overline{\Omega} \times \R$. Suppose that there is a sequence $(x_n, y_n, t_n) \in \overline{\Omega} \times \R$ such that $f_s(x_n, y_n, t_n) \to +\infty$. Since we have 
\begin{align}\label{e:bound01}
0\leqslant \frac{y^2}{(t-x)^2+y^2} \leqslant 1, \quad \forall (x,y,t) \in \R\times \R^*\times \R,
\end{align}
and since $s>0$, we infer that $\langle t_n\rangle /\langle x_n \rangle\to \infty$, as $n\to \infty$.  In particular, since $s\leq 1$, we have $(\langle t_n\rangle /\langle x_n \rangle)^s \leqslant \langle t_n\rangle /\langle x_n \rangle$ for $n$ large enough and $|t_n|\to \infty$, as $n\to \infty$. 
Next, we have:
\begin{align}\label{e:festi}
0\leqslant f_s(x_n, y_n, t_n)&\leqslant \frac{\langle t_n\rangle ^2}{\langle x_n\rangle ^2}\frac{\frac{y^2_n}{\langle x_n\rangle^2}}{\frac{\langle t_n\rangle ^2}{\langle x_n\rangle ^2} \left(\frac{t_n}{\langle t_n\rangle}-\frac{x_n}{\langle t_n\rangle}\right)^2+\frac{y^2_n}{\langle x_n\rangle^2}} \leqslant \frac{c}{\left(\frac{t_n}{\langle t_n\rangle}-\frac{x_n}{\langle t_n\rangle}\right)^2}.
\end{align}
We infer:
\[\limsup_{n\to \infty } f_s(x_n, y_n, t_n) \leqslant c,\]
which is a contradiction. \qed

\begin{Lemma}
\label{infamousLemma}
Let $Q$ and $\Omega$ be as in Lemma \ref{infamousLemma1}. For every $0 < s < 1$ and $p \in \R$ there exists $C > 0$ such that for all $z = x+\i y \in \Omega$ :
\begin{equation}
\big \| \ln ^{p} \left(1+\langle Q \rangle \right) \langle Q \rangle ^s (Q-z)^{-1} \big \| \leqslant C \ln ^{p} \left(1+\langle x \rangle \right) \langle x \rangle ^s  |y|^{-1}.
\label{norm2}
\end{equation}
\end{Lemma}

\begin{proof}
The case $p=0$ is covered by the previous lemma. By the spectral theorem, 
$$\big \| \ln ^{p} \left(1+\langle Q \rangle \right) \langle Q \rangle ^s (Q-z)^{-1} \big \| ^2 \leqslant \sup_{t \in \R} \ln ^{2p} \left( 1 + \langle t \rangle \right) \langle t \rangle ^{2s} \left( (t-x)^2+y^2 \right)^{-1}.$$
Therefore, we consider the function 
\begin{equation*}
g_s(x,y,t) : =  \frac{\ln ^{2p} \left( 1 + \langle t \rangle \right)}{\ln ^{2p} \left( 1 + \langle x \rangle \right)}\frac{\langle t \rangle ^{2s}}{\langle x \rangle ^{2s}}\frac{y^2}{(t-x)^2+y^2}, \quad \text{defined on} \ (x,y,t) \in \overline{\Omega} \times \R \setminus \Lambda,
\end{equation*}
where $\overline{\Omega} = \{ (x,y) \in \R^2 : |y| \leqslant c \langle x \rangle \}$ and $\Lambda := \{ (x,y,t) \in \overline{\Omega} \times \R : y=0, x=t\}$. 
We aim at proving that $g_s$ is uniformly bounded on the entire region $\overline{\Omega} \times \R$. Suppose that there is a sequence $(x_n, y_n, t_n) \in \overline{\Omega} \times \R$ such that $g_s(x_n, y_n, t_n) \to +\infty$. Recalling \eqref{e:bound01}, since $s>0$, up to a subsequence, we infer that 
\begin{align*}
\text{ either } &1)\, \frac{\ln ^{2p} \left( 1 + \langle t_n \rangle \right)}{\ln ^{2p} \left( 1 + \langle x_n \rangle \right)} \to \infty, 
\\
\text{ or } &2)\, \frac{\langle t_n\rangle}{\langle x_n \rangle}\to \infty,
 \end{align*}
as $n\to \infty$. We start with case 2). Given $\alpha>0$, there are $C_\alpha, C_\alpha'>0$ such that $\forall a,b\in \R$ we have
\begin{align}
\label{e:logesti}
0\leqslant \frac{\ln \left( 1 + \langle a \rangle \right)}{\ln \left( 1 + \langle b \rangle \right)}&= \frac{\ln\left( \frac{1+\langle a\rangle}{1+\langle b\rangle}\right)}{\ln(1+\langle b\rangle)} +1 \leqslant \frac{1}{\ln(2)}\ln\left( \frac{1+\langle a\rangle}{1+\langle b\rangle}\right)+1
\\
\nonumber
& \leqslant C_\alpha \left( \frac{1+\langle a\rangle}{1+\langle b\rangle}\right)^{\alpha} \leqslant C_\alpha' \left( \frac{\langle a\rangle}{\langle b\rangle}\right)^{\alpha}. 
\end{align}
Recalling $s\in (0,1)$ and by considering the cases $p>0$ and $p<0$ separately, by choosing $\alpha$ small enough, there is a finite constant $C$ such that 
\[0 \leqslant g_s(x_n, y_n, t_n)\leqslant C f_1(x_n, y_n, t_n),\]
where $f_s$ is the function in Lemma \ref{infamousLemma1}. Repeating \eqref{e:festi}, we obtain a contradiction. 

We turn to the case 1). If $p>0$, \eqref{e:logesti}
ensures that $\frac{1+\langle t_n \rangle}{1+\langle x_n \rangle}\to \infty$, as $n\to \infty$. 
Hence the case 2) holds which in turn gives a contradiction.

Suppose now that 1) holds and $p<0$. Unlike before, we obtain this time that $\frac{\langle x_n\rangle}{\langle t_n \rangle}\to \infty$, as $n\to \infty$. Using \eqref{e:bound01} and \eqref{e:logesti}, there is $C, s' >0$ such that
\begin{align*}
0 \leqslant g_s(x_n, y_n, t_n) \leqslant C \left(\frac{\langle t_n\rangle}{\langle x_n \rangle}\right)^{s'}\to 0,
\end{align*}
as $n\to \infty$. This is a contradiction. 
\qed
\end{proof}

\begin{Lemma}
\label{infamousLemma3}
Let $Q$ and $\Omega$ be as in Lemma \ref{infamousLemma1}. For every $0 < s < 1$ and $\{p_0,...,p_{m}, p_{m+1}\} \subset \R$ there exists $C > 0$ such that for all $z = x+\i y \in \Omega$ :
\begin{equation}
\Big \|  \langle Q \rangle ^s  (Q-z)^{-1} \ln_{m+1} ^{p_{m+1}} \left( \langle Q \rangle \right) \prod_{k=0} ^m \ln_{k} ^{p_k} \left( \langle Q \rangle \right) \Big \| \leqslant C  \langle x \rangle ^s  |y|^{-1} \ln_{m+1} ^{p_{m+1}} \left( \langle x \rangle \right) \prod_{k=0} ^m \ln_{k} ^{p_k} \left( \langle x \rangle \right).
\label{norm5}
\end{equation}
\end{Lemma}
\begin{proof}
Thanks to \eqref{e:logesti}, the proof is analogous to that of Lemma \ref{infamousLemma}.
\qed
\end{proof}

We now display Thierry Jecko's more general result. Let $I = [1, +\infty)$. We assume 3 things:

\begin{itemize}
\item A1: Let $\psi_d : I \mapsto \R^{+} \setminus \{0\}$ be a continuous and decreasing (non-increasing) function such that for all $\epsilon >0$, the function $I \ni x \mapsto (\psi_d (x)  + 1/\psi_d (x) ) x^{-\epsilon}$ is bounded.

\item A2: Let $\psi_u : I \mapsto \R^{+} \setminus \{0\}$ be a continuous and increasing (non-decreasing) function such that for all $\epsilon >0$, the function $I \ni x \mapsto (\psi_u (x)  + 1/\psi_u (x) ) x^{-\epsilon}$ is bounded. 

\item A3: $\psi_u$ is sub-multiplicative, that is, there is $M >0$ such that for all $(x,y) \in I^2$, $\psi_u (xy) \leq M \psi_u (x) \psi_u (y)$. 
\end{itemize}
Now let $s \in (0,1)$ and define $f : I  \mapsto \R^{+} \setminus \{0\}$ by $f(x) := x^{2s} \psi_u (x) \psi_d (x)$.

\begin{proposition}
\label{Proposition ThierryJecko} Let $Q$ and $\Omega$ be as in Lemma \ref{infamousLemma1}. In addition to the assumptions A1, A2 and A3 concerning $\psi_d$ and $\psi_u$, suppose there is $R_0 \geq 1$ such that $f$ is increasing on $[R_0, + \infty)$. Then there is $C>0$ such that $\| \sqrt{f ( \langle Q \rangle)} (Q -z)^{-1} \| \leq C \sqrt{ f(\langle x \rangle)} |y|^{-1}$ uniformly in $z=x+\i y \in \Omega$. 
\end{proposition}
\begin{remark}
Letting $\psi_u (x) =  \prod_{k=0, p_k >0} ^{m+1} \ln_{k} ^{2p_k} \left( x \right)$ and $\psi_d (x) =  \prod_{k=0, p_k < 0} ^{m+1} \ln_{k} ^{2p_k} \left( x \right)$ immediately yields Lemma \ref{infamousLemma3}. On a side note, to satisfy the assumption A3 it is probably also necessary to have a larger constant in the Napierian logarithm, i.e.\ instead of $\ln_1(1+x) := \ln(1+x)$ and $\ln_k (x) := \ln(1 + \ln_{k-1}(x))$ one should take $\ln_1(\mu+x) := \ln(\mu+x)$ and $\ln_k (x) := \ln(\mu + \ln_{k-1}(x))$ with $\mu \gg 0$. Such adjustment is inconsequential in this article.
\end{remark}
\begin{proof}
First we choose an analytic extension of $f$ such that $c \leq 1/4$ in the definition of $\Omega$ in Lemma \ref{DGDavies}.

In this proof $C>0$ denotes a generic constant that does not depend on $x,y$ and $t$. Let $\Omega_R := \{ (x,y) \in \R^2 : 0 < |y| \leqslant c \langle x \rangle \leq R \}$. Let $f_u (x) := x^{2s} \psi_u (x)$. Then $f_u$ is sub-multiplicative, increasing on $I$ and $x\mapsto x^{-2} f_u(x)$ is bounded by A2. For $y\neq 0$, let 
$$g(t) := \frac{f( \langle t \rangle) }{(t-x)^2+y^2}, \quad \text{and} \quad h(t) := \frac{f( \langle t \rangle) }{(t-x)^2 y^{-2}+1}.$$
Of course, $g(t) = h(t) y^{-2}$. To prove the proposition, it is enough by functional calculus to prove that there is $C >0$ such that for all $(x,y) \in \Omega$, $\sup_{t \in \R} g(t) \leq C f(\langle x \rangle ) y^{-2}$, or equivalently, 
\begin{equation}
\label{want_to_show}
\sup_{t \in \R} h(t) \leq C f(\langle x \rangle ).
\end{equation} 
First we note that assumptions A1 and A2 imply that $f(x) = x^{2s} \psi_u (x) \psi_d (x) \geq C_{\epsilon} x^{2s} x^{-2\epsilon}$ for any $\epsilon >0$ and so $\lim f(x) = +\infty$ as $x \to +\infty$. In turn, this implies (together with the fact that $f$ is continuous) that there is $R_1 \geq R_0$ such that  $\max _{1 \leq x' \leq R_0} f(x') \leq f(x)$ for all $x \geq R_1$. In Part~1, the estimate \eqref{want_to_show} is proved for $(x,y) \in \Omega_R$ for any finite $R$ (with the constant $C$ depending on $R$). In Part 2, the estimate \eqref{want_to_show} is proved for $(x,y) \in \Omega \setminus \Omega_{2R_1}$. 

\noindent \underline{{\footnotesize PART 1}:} Let $R > 1$ and $(x,y) \in \Omega_R$ be arbitrarily given. For $|t|  \geq 4R$, $|t-x| \geq |t| /2$ and $(t-x)^2 y^{-2} + 1 \geq t^2 / (4y^2) + 1 \geq t^2 /(4y^2)$. So
$$h(t) \leq \frac{f( \langle t \rangle )}{t^2} 4y^2 \leq 4R^2 \frac{f( \langle t \rangle )}{\langle t \rangle^2} \frac{\langle t \rangle^2}{t^2} \leq 4R^2 \bigg \| \frac{f( \langle t \rangle )}{\langle t \rangle^2} \bigg \|_{\infty} \bigg \| \frac{\langle t \rangle^2}{t^2} \bigg \|_{\infty} \leq C,$$
where $C$ grows with $R$. By continuity of the function $h$ one gets $h(t) \leq C$ uniformly in $t \in \R$.  Finally, $1/f$ being uniformly bounded by assumption,  $\sup_{t\in \mathbb{R}}h(t) \leq  C \left\| \frac{1}{f(\langle x \rangle )} \right\|_{\infty} \cdot f( \langle x \rangle )$.

\noindent \underline{{\footnotesize PART 2}:} Let $(x,y) \in \Omega$ with $|x| \geq 2 R_1$. 

$\bullet$ Case 1: For $t \in \R$ satisfying $|t-x| \leq |x|/2$, one has $\langle t \rangle ^2 = 1 + t^2 \leq 1 + 2(t-x)^2 + 2x^2 \leq 4 \langle x \rangle ^2$, and so $\langle t \rangle \leq 2 \langle x \rangle$. Also, $|t | \geq |x| - |x-t| \geq |x| /2 \geq R_1$. Since $2 \langle x \rangle \geq \langle t \rangle \geq R_1$ and $f$ is increasing on $[R_0, \infty)$, it follows that 
$$h(t) \leq f( \langle t \rangle ) \leq  f( 2 \langle x \rangle ) = f_u (2 \langle x \rangle ) \psi_d (2 \langle x \rangle ) \leq C f_u (\langle x \rangle ) \psi _d (\langle x \rangle) = C  f (\langle x \rangle ).$$
We used that $\psi_d$ is decreasing and  $\psi_u$ is  sub-multiplicative  in the last inequality. 

$\bullet$ Case 2: For $t \in \R$ satisfying $|t-x| \geq |x|/2$, one has 
\[\langle t \rangle ^2 \leq 1 + 2(t-x)^2 + 2x^2 \leq 1 + 2(t-x)^2 + 8 (t-x)^2 \leq 16 \langle t-x \rangle ^2.\] One further subdivides into 2 subcases. 

\tiny $\bullet$ \normalsize  Subcase 2.1: If $|t | \leq |x|$, then $\langle x \rangle \geq \max(2 R_1, \langle t \rangle)$. So,
$$h(t) \leq f(\langle t \rangle) \leq \max_{\langle t' \rangle \leq R_0} f( \langle t' \rangle ) + \max_{R_0 \leq \langle t' \rangle \leq \langle x\rangle} f(\langle t' \rangle) \leq 2 f(\langle x \rangle),$$
where we used the assumption about $R_1$ and that $f$ is increasing on $[R_0, \infty)$. 

\tiny $\bullet$ \normalsize  Subcase 2.2: If $|t| \geq |x|$, then $2R_1 \leq |x | \leq \langle x\rangle \leq \langle t \rangle \leq 4 \langle t-x \rangle$. So $f (\langle t \rangle ) \leq f(4 \langle t -x \rangle )$. This leads to $h(t) \leq f(4 \langle t -x \rangle ) / (1+(t-x)^2 y^{-2})$. Next, since $c \leq 1/4$ and $R_1\geq 1$, one has 
\begin{align*}
|y| \leq |y| / (4c) \leq  \langle x \rangle /4 \leq |x| / 2 \leq |t-x| \text{ and then } \frac{\langle t-x \rangle}{\langle y \rangle} \in I.
\end{align*}
Since $\psi_d$ is decreasing, $\langle x \rangle \leq 4 \langle t-x \rangle$, $\psi_d$ is decreasing, and $f_u$ is  sub-multiplicative, we have:
$$ h(t) \leq \frac{f_u(4 \langle t -x \rangle ) \psi_d(4 \langle t -x \rangle )}{1+(t-x)^2 y^{-2}} \leq C \frac{f_u( \langle t -x \rangle )  \psi_d(\langle x \rangle )}{1+(t-x)^2 y^{-2}}  \leq C f_u \left( \frac{\langle t-x \rangle}{\langle y \rangle} \right)  \frac{f_u ( \langle y \rangle )  \psi_d(\langle x \rangle )}{1+(t-x)^2 y^{-2}} .$$
Finally, since we have $\langle t-x \rangle \leq \langle (t-x)/y \rangle \cdot \langle y \rangle$, $1+(t-x)^2 y^{-2} = \langle (t-x)/y \rangle ^2$, and $\langle y \rangle \leq  \langle x \rangle$, we conclude:
$$h(t) \leq C \frac{f_u(\langle (t-x)/y \rangle )}{ \langle (t-x)/y \rangle ^2} f_u ( \langle x \rangle ) \psi_d ( \langle x \rangle ) \leq C \bigg \| \frac{f_u(\langle (t-x)/y \rangle )}{ \langle (t-x)/y \rangle ^2} \bigg \|_{\infty} f_u ( \langle x \rangle ) \psi_d ( \langle x \rangle ) \leq C f ( \langle x \rangle ).$$
 \qed
\end{proof}

\subsection{A Heinz inequality for logarithmic functions}
\label{SubsectionHeinz}
Let $T,S$ be arbitrary self-adjoint and bounded from below in a Hilbert space $\mathscr{H}$. The notation $T \leqslant S$ means that $\mathrm{Dom}[|S|^{1/2}]\subset \mathrm{Dom}[|T|^{1/2}]$ and $\langle \psi, T \psi \rangle \leqslant \langle \psi, S \psi \rangle$ for all $\psi \in  \mathrm{Dom}[|S|^{1/2}]$. The \textit{Heinz inequality}, cf.\ \cite{Ka1, H}, states 
\begin{equation}
\label{HEINZ}
 0 \leqslant T \leqslant S \Rightarrow T^{\alpha} \leqslant S^{\alpha}, \quad \text{for all} \ 0 \leqslant \alpha \leqslant 1,
 \end{equation}
The aim of this sub-section is to extend this transitive property to a family of (poly)logarithmic functions. In what follows we consider general positive self-adjoint operators $T,S$ and derive results in greater generality. Concerning the detailed application we make of it, notably for the proof of Theorem \ref{lapy3}, please see sub-section \ref{subsection_explain}. 

In this sub-section we now always suppose $1 \leqslant T \leqslant S$ and hence  Dom$[S^{1/2}] \subset$ Dom$[T^{1/2}]$. By the Heinz inequality, $T ^{2\alpha} \leqslant S ^{2\alpha}$, $0 \leqslant \alpha \leqslant 1/2$. But this is equivalent to saying that $\| T ^{\alpha  } \psi \| \leqslant \| S ^{\alpha  } \psi \|$ for all $\psi \in \mathrm{Dom}[ S^{\alpha}]$. Since $S^{\alpha}$ is bijective, it follows that $\| T^{\alpha} S^{-\alpha} \psi \| \leqslant \| \psi \|$, for all $\psi \in \mathscr{H}$. Thus $T ^{\alpha} S ^{-\alpha}$ is an element of $\mathscr{B}(\mathscr{H})$. In what follows this small argument is used repeatedly. The Helffer-Sj\"ostrand formula is also used, see Appendix \ref{Helffer-Appendix}. Let $\Phi_n$, $n \in \N^*$, be the polylogarithmic functions from \eqref{PHI}. Denote $\Phi_n ^{\alpha} (x) := ( \Phi_n (x) )^{\alpha}$.

\begin{Lemma} 
\label{cor1}
For all $0 \leqslant \alpha \leqslant 1/2$, $\Phi_n ^{\alpha} (T)  \cdot \Phi_n ^{-\alpha} (S) \in \mathscr{B}(\mathscr{H})$, $n \in \N^*$. 
\end{Lemma}
\begin{proof}
As discussed in Section \ref{Polylog} the $\Phi_n$, $n \in \N^*$, are Nevanlinna functions that continue analytically across $(-1,+\infty)$ from $\C_+$ to $\C_-$. Since $1 \leqslant T \leqslant S$, Theorem \ref{Prop33} implies $\Phi_n (T) \leqslant \Phi_n (S)$. By the Heinz inequality, $\Phi_n ^{2\alpha}  (T) \leqslant  \Phi_n ^{2\alpha} (S)$, $0 \leqslant \alpha \leqslant 1/2$. The result follows.
\qed
\end{proof}

\begin{Lemma} 
\label{cor22}
For all $k \in \N$ and $0 \leqslant \alpha \leqslant 1/2$, $\Phi_n ^{\alpha} (\ln_k(T)) \cdot \Phi_n ^{-\alpha}  (\ln_k(S)) \in \mathscr{B}(\mathscr{H})$, $n \in \N^*.$
\end{Lemma}
\begin{proof}
$\Phi_n(\ln_k(x))$ belongs to $P(-1,+\infty)$ because it is the composition of functions all belonging to $P(-1,+\infty)$. Also, $1 \leqslant T \leqslant S$. By Loewner's Theorem, $\Phi_n(\ln_k(T)) \leqslant \Phi_n (\ln_k(S))$. The result follows from the Heinz inequality.
\qed
\end{proof}

\begin{Lemma}
\label{LogAtoN}
For all $0 \leqslant p$, $\ln^{p} (1+T) \cdot \ln^{-p} (1+S) \in \mathscr{B}(\mathscr{H}).$
\end{Lemma}

\begin{remark}
For $0 \leqslant p \leqslant 1/2$, the result follows from well-known results on the logarithm, see e.g.\ \cite[Exercise 51 of Chapter VIII]{RS1}. For $p>1/2$ we don't know how prove the result without the use of the polylogarithmic functions.
\end{remark}

\begin{proof} 
Write $p = n\alpha$, where $n \in \N^*$ and $0 \leqslant \alpha \leqslant 1/2$. Then $\ln^{n \alpha} (1+T) \ln^{-n\alpha} (1+S)$ equals
$$\underbracket{\ln^{n \alpha} (1+T) \Phi_n ^{-\alpha}(T) }_{\text{bounded by } \eqref{limits Li}} \underbracket{\Phi_n ^{\alpha}(T)\Phi_n ^{-\alpha}(S) }_{\text{bounded by Lemma } \ref{cor1}} \underbracket{ \Phi_n ^{\alpha} (S) \ln ^{-n\alpha} (S) }_{\text{bounded by } \eqref{limits Li}}.$$ 
\qed
\end{proof}
The next lemma extends the previous to iterated logarithms.
\begin{Lemma}
\label{LogAtoN2}
For all $k \in \N$ and $0 \leqslant p$, $\ln_k ^{p} \left(T \right) \cdot \ln_k ^{-p} \left( S \right) \in \mathscr{B}(\mathscr{H}).$
\end{Lemma}
\begin{proof}
The statement is trivial for $k=0$ and $k=1$ is Lemma \ref{LogAtoN}. Now let $k \geqslant 1$. Write $p = n\alpha$, where $n \in \N^*$ and $0 \leqslant \alpha \leqslant 1/2$. Then  $\ln_k^{n\alpha} \left(T \right) \ln_k^{-n\alpha} \left(S \right)$ equals
$$\underbracket{\ln_k^{n\alpha} (T) \Phi_n ^{-\alpha}(\ln_{k-1}(T)) }_{\text{bounded by } \eqref{limits Li}} \underbracket{\Phi_n ^{\alpha}(\ln_{k-1}(T))\Phi_n ^{-\alpha}(\ln_{k-1}(S)) }_{\text{bounded by Lemma } \ref{cor22}} \underbracket{ \Phi_n ^{\alpha} (\ln_{k-1}(S)) \ln_k ^{-n\alpha} \left( S \right) }_{\text{bounded by } \eqref{limits Li}}.$$
\qed
\end{proof}

The next Lemma extends the previous to products of iterated logarithms. We require extra technical regularity or a much more specific assumption. 



\begin{Lemma}
\label{LogAtoN_general}
Let $\{p_0,...,p_{m}, p_{m+1}\} \subset [0,+\infty)$ and $m \in \N$ be given. Suppose either
\begin{itemize}
\item (a)
$\prod_{k=0} ^{m+1} \ln_{k} ^{-p_k} \left(S \right) \in \mathcal{C}^1(T)$, or
\item (b)
$T = \langle Q \rangle$ for some self-adjoint operator $Q$ and $\prod_{k=0} ^{m+1} \ln_{k} ^{-p_k} \left(S \right) \in \mathcal{C}^1(Q)$. 
\end{itemize}
Then 
$$\prod_{k=0} ^{m+1} \ln_{k} ^{p_k} \left( T \right) \cdot \prod_{k=0} ^{m+1} \ln_{k} ^{-p_k} \left( S \right)   \in \mathscr{B}(\mathscr{H}).$$
\end{Lemma}
\begin{proof}
By induction on $m$. The base case $m=0$ is Lemma \ref{LogAtoN} (recall that $\ln_1(x) := \ln(1+x)$). For the inductive step, $m \to m+1$, we have 
$$\ln_{m+2} ^{p_{m+2}} \left( T \right) \prod_{k=0} ^{m+1} \ln_{k} ^{p_k} \left( T \right) \cdot \ln_{m+2} ^{-p_{m+2}} \left( S \right) \prod_{k=0} ^{m+1} \ln_{k} ^{-p_k} \left( S \right).$$
After commuting $\ln_{m+2} ^{p_{m+2}} \left( T \right)$ with $\prod_{k=0} ^{m+1} \ln_{k} ^{-p_k} \left( S \right) $ this operator is equal to :
$$\underbracket{ \prod_{k=0} ^{m+1} \ln_{k} ^{p_k} \left( T \right) \prod_{k=0} ^{m+1} \ln_{k} ^{-p_k} \left( S \right)}_{\text{bounded by the Induction hypothesis}} \cdot \underbracket{ \ln_{m+2} ^{p_{m+2}} \left( T \right)  \ln_{m+2} ^{-p_{m+2}} \left( S \right) }_{\text{bounded by Lemma \ref{LogAtoN2} }} $$
$$ + \underbracket{ \prod_{k=0} ^{m+1} \ln_{k}^{p_k} \left( T \right) \Big [ \ln_{m+2} ^{p_{m+2}} \left( T \right), \prod_{k=0} ^{m+1} \ln_{k} ^{-p_k} \left(S \right) \Big ]_{\circ} }_{\text{to be developed}} \ln_{m+2} ^{-p_{m+2}} \left(S \right).$$
It is enough to show that the latter term with the under bracket is a bounded operator. If we are assuming $(a)$ then the term is developed as follows :
$$\frac{\i}{2\pi} \int_{\C} \frac{\partial \tilde{f}}{\partial \overline{z}} \prod_{k=0} ^{m+1} \ln_{k} ^{p_k} \left( T \right) (z-T)^{-1} \Big [T, \prod_{k=0} ^{m+1} \ln_{k} ^{-p_k} \left( S \right) \Big ]_{\circ} (z-T)^{-1} \dz,$$
where $f(x) = \ln ^{p_{m+2}} _{m+2}\left(x \right) \in \mathcal{S}^{\epsilon}(\R), \forall \epsilon >0$, see \eqref{decay1}. By the assumption $(a)$ the commutator in the middle of the integral $\in \mathscr{B}(\mathscr{H})$. To keep things simple, one can use Lemma \ref{infamousLemma1} to get 
$$ \Big \| \prod_{k=0} ^{m+1} \ln_{k} ^{p_k} \left( T \right) (z - T)^{-1} \Big \| \leqslant c \big \| \langle T \rangle ^{\delta} (z - T)^{-1} \big \| \leqslant c \langle x \rangle ^{\delta} |y|^{-1}$$
for some $\delta >0$ and $c>0$. Thus the latter integral converges in norm to a bounded operator. If we are assuming $(b)$ however, then the term is developed as follows :
$$\frac{\i}{2\pi} \int_{\C} \frac{\partial \tilde{h}}{\partial \overline{z}} \prod_{k=0} ^{m+1} \ln_{k} ^{p_k} \left( \langle Q \rangle \right) (z-Q)^{-1} \Big [Q, \prod_{k=0} ^{m+1} \ln_{k} ^{-p_k} \left( S \right) \Big ]_{\circ} (z-Q)^{-1} \dz,$$
where $h(x) = \ln ^{p_{m+2}} _{m+2}\left( \langle x \rangle \right) \in \mathcal{S}^{\epsilon}(\R), \forall \epsilon >0$. By the assumption $(b)$ the commutator in the middle of the integral $\in \mathscr{B}(\mathscr{H})$. The rest is as before.
\qed
\end{proof}

\subsection{From a LAP in $T$ to a LAP in $S$}
\label{1to1LAP}

In this sub-section we continue assuming $T,S$ are arbitrary self-adjoint operators satisfying $1 \leqslant T \leqslant S$. Suppose we have a LAP of the form \eqref{LAP_this_article2} and a local decay estimate \eqref{local_decay2} where the weights are in $T$. Here we derive estimates that allow to pass to those same estimates with weights in $S$ instead. We state results in greater generality. Concerning the detailed application we make of it please see sub-section \ref{subsection_explain}.  

The next Lemma extends Lemma \ref{LogAtoN}.
\begin{Lemma} 
\label{weightsAtoN}
Fix $\sigma \in [0,1/2]$ and suppose either 
\begin{itemize}
\item (a)
$S^{-\sigma} \in \mathcal{C}^1(T)$, or
\item (b)
$T = \langle Q \rangle$ for some self-adjoint operator $Q$ and $S^{-\sigma} \in \mathcal{C}^1(Q)$. 
\end{itemize}
Then for all $0 \leqslant p$,  $ T ^{\sigma} \ln^{p} (1+ T) \cdot S ^{-\sigma}  \ln^{-p} (1+S)  \in \mathscr{B}(\mathscr{H})$. \end{Lemma}
\begin{proof}
The case $\sigma = 0$ is covered by Lemma \ref{LogAtoN} so we assume $\sigma \in (0,1/2]$. Since $T ^{\sigma} S^{-\sigma} \in \mathscr{B}(\mathscr{H})$ and $\ln^{p} (1+T) \ln^{-p} (1+S) \in \mathscr{B}(\mathscr{H})$, it is enough to prove that $T ^{\sigma}  [\ln^{p} (1+T), S ^{-\sigma}] _{\circ} \in \mathscr{B}(\mathscr{H})$. As in the proof of Lemma \ref{LogAtoN_general} we proceed slightly differently depending on whether $(a)$ or $(b)$ is assumed. Let us treat the case $(a)$ only. The formal commutator $T ^{\sigma}  [\ln^{p} (1+T), S ^{-\sigma}]$ is equal to 
$$\frac{\i}{2\pi} \int_{\C} \frac{\partial \tilde{f}}{\partial \overline{z}} T ^{\sigma} (z-T)^{-1} [T, S ^{-\sigma} ]_{\circ} (z-T)^{-1} \dz,$$
where $f(x) = \ln ^p(1+x) \in \mathcal{S}^{\epsilon}(\R), \forall \epsilon >0$. By assumption $[T, S^{-\sigma}]_{\circ} \in \mathscr{B}(\mathscr{H})$. Applying Lemma \ref{infamousLemma1} shows that this integral converges in norm to a bounded operator.
\qed
\end{proof}

The next Lemma extends Lemma \ref{LogAtoN_general}.

\begin{Lemma} 
\label{weightsAtoN_general}
Let $\sigma \in [0,1/2]$, $\{p_0,...,p_{m}, p_{m+1}\} \subset [0,+\infty)$ and $m \in \N$ be given. Suppose either 
\begin{itemize}
\item (a)
$S^{-\sigma} \prod_{k=0} ^{m+1} \ln_{k} ^{-p_k} \left(S \right) \in \mathcal{C}^1(T)$, or
\item (b)
$T = \langle Q \rangle$ for some self-adjoint operator $Q$ and $S^{-\sigma} \prod_{k=0} ^{m+1} \ln_{k} ^{-p_k} \left(S \right)  \in \mathcal{C}^1(Q)$. 
\end{itemize}
Then
$$ T ^{\sigma}  \prod_{k=0} ^{m+1} \ln_{k} ^{p_k} \left( T \right)  \cdot S ^{-\sigma}  \prod_{k=0} ^{m+1} \ln_{k} ^{-p_k} \left( S \right)  \in \mathscr{B}(\mathscr{H}).$$
\end{Lemma}
\begin{proof}
By induction on $m$. The base case $m=0$ is Lemma \ref{weightsAtoN}. The inductive step is handled in the same way as in Lemma \ref{LogAtoN_general}. One commutes $\ln_{m+2} ^{p_{m+2}} \left( T \right)$ with $S ^{-\sigma} \prod_{k=0} ^{m+1} \ln_{k} ^{-p_{k}} \left( S \right) $. One applies the inductive hypothesis to the first term, and for the second term (the one with the commutator) it is enough to show that
$$ T ^{\sigma} \prod_{k=0} ^{m+1} \ln_{k} ^{p_k} \left( T \right)  \Big[ \ln_{m+2} ^{p_{m+2}} \left( T \right) , S^{-\sigma} \prod_{k=0} ^{m+1} \ln_{k} ^{-p_k} \left(S \right)  \Big]_{\circ} \in  \mathscr{B}(\mathscr{H}).$$
To this end one performs a similar integral expansion as in the proof of Lemma \ref{LogAtoN_general}, with the same functions, $f(x) = \ln_{m+2} ^{p_{m+2}} (x)$ or $h(x) = \ln_{m+2} ^{p_{m+2}} (\langle x\rangle)$ depending on whether $(a)$ or $(b)$ is assumed. Then, to keep things simple, one can use the fact that there are $\delta >0$ and $c>0$ (by Lemma \ref{infamousLemma1}) such that
$$ \Big \| T ^{\sigma} \prod_{k=0} ^{m+1} \ln_{k} ^{p_k} \left(T \right) (z - T)^{-1} \Big \| \leqslant c \big \| \langle T \rangle ^{\sigma + \delta} (z - T)^{-1} \big \| \leqslant c \langle x \rangle ^{\sigma+\delta} |y|^{-1}.$$
One sees that the integral converges in norm to a bounded operator.
\qed
\end{proof}

\subsection{The 2 previous sub-sections in application} 
\label{subsection_explain}

We are now in a position to explain how the results of the preceding 2 sub-sections are applied in the proof of Theorem \ref{lapy3}.

Let $A$ and $N$ be respectively the discrete generator of dilations and the position operator. $\langle A \rangle$ and $\langle N \rangle$ are unbounded self-adjoint operators. The unboundedness is seen directly from the graph norm and using an appropriate sequence of unit vectors. 

The goal is to apply the results of sub-sections \ref{SubsectionHeinz} and \ref{1to1LAP}, chiefly Lemmas \ref{LogAtoN_general} and \ref{weightsAtoN_general}, to $(T,S) =(\langle A \rangle, \langle N \rangle)$. However, this cannot be done as such. Instead what we do is to first establish the inequality $1 \leqslant \langle A \rangle \leqslant \sqrt{c_d} \langle N \rangle$, see Lemma \ref{A<N} below. We then apply Lemmas \ref{LogAtoN_general} and \ref{weightsAtoN_general} to $(T,S) =(\langle A \rangle, \sqrt{c_d} \langle N \rangle)$. Only then do we see that their conclusions remain valid for $(T,S) =(\langle A \rangle, \langle N \rangle)$.

For our Schr\"odinger problem, in application we shall need the conclusion of Lemma \ref{LogAtoN_general} to hold for $(T,S) = (\langle A \rangle, \langle N \rangle)$, $p_k \equiv 1$ for $0 \leqslant k \leqslant m$ and $p_{m+1} >1$, see for example Lemma \ref{Lemma_compact_short_range}. To pass from a LAP with weights in $\langle A \rangle$ to weights in $\langle N \rangle$, we shall also need the conclusion of Lemma \ref{weightsAtoN_general} to hold for $(T,S) = (\langle A \rangle, \langle N \rangle)$, $\sigma = 1/2$, $p_k \equiv 1/2$ for $0 \leqslant k \leqslant m$ and $p_{m+1} >1/2$. 



In the 2 previous sub-sections we have kept the options open for 2 assumptions : $(a)$ or $(b)$, cf.\ Lemmas \ref{LogAtoN_general} and \ref{weightsAtoN_general}. In application we choose $(b)$, i.e.\ we set $T = \langle Q \rangle$ where $Q=A$, the generator of dilations. Thus we must verify the regularity criteria :

\begin{Lemma}
\label{regularityA_N}
Let $\sigma \in [0,1/2]$, $\{p_0,...,p_{m}, p_{m+1}\} \subset [0,+\infty)$ and $m \in \N$ be arbitrary. Then
$\langle N \rangle ^{-\sigma} \prod_{k=0} ^{m+1} \ln_{k} ^{-p_k} \left(\langle N \rangle  \right) \in \mathcal{C}^1( A )$.
\end{Lemma}
\begin{proof}
This is a simple application of Proposition \ref{Prop629}, \eqref{commutatorV} and the mean value theorem.
\qed
\end{proof}

Finally, we must give the key inequality that allows to apply Lemmas \ref{LogAtoN_general} and \ref{weightsAtoN_general} to $(T,S) =(\langle A \rangle, \sqrt{c_d} \langle N \rangle)$ : 

\begin{Lemma} 
\label{A<N}
Let $c_d := \max(d^3+d^2+1, 4(d+1))$. $\forall \ 0 \leqslant \alpha \leqslant 1$, $(A^2+1)^{\alpha} \leqslant (c_d)^{\alpha} (N^2+1)^{\alpha}$. 
\end{Lemma}
\begin{proof}
The result is not new but we repeat the proof for convenience. Let $\psi \in \mathrm{Dom}[N^2] \subset \mathrm{Dom}[A^2]$. First we calculate for $\alpha = 1$ :
\begin{align*}
\big \langle \psi , (A^2+1) \psi \big \rangle &= \| \psi \|^2 + \| A \psi \|^2 \leqslant \| \psi \|^2 + \bigg[ \sum_{j=1} ^d \| \psi \| + 2 \| N_j \psi \| \bigg]^2 \\
& \leqslant  \| \psi \|^2 + \bigg[ d^2 \| \psi \| ^2 + \sum_{j=1} ^d  4 \| N_j \psi \|^2 \bigg] (d+1) \leqslant c_d \bigg[ \| \psi \|^2 + \sum_{j=1} ^d \| N_j \psi \|^2  \bigg] \\
&= c_d  \ \big \langle \psi , (N^2+1) \psi \big \rangle. 
\end{align*}
Pass to exponent $\alpha$ by invoking the Heinz inequality \eqref{HEINZ}.
\qed
\end{proof}

\section{Preliminary lemmas for the Proof of Theorem \ref{lapy3}}
\label{lemmaSection}


Let $P$ denote the orthogonal projection onto the pure point spectral subspace of $H$. $A$ denotes the discrete generator of dilations \eqref{generatorDilations}. Please note that throughout all this Section we state the results for the Hamiltonian $H$, but they also apply to $H_2$ with the same proof (just exchange $W$ with $W_2$ and $\tilde{W}$ with $\sigma$).

\begin{proposition} \cite{GJ1} $\forall u,v \in \mathrm{Dom}[A]$, the rank one operator $|u \rangle \langle v |: \psi \to \langle v, \psi \rangle u$ $\in \mathcal{C}^1(A)$.
\label{rank one}
\end{proposition}

\begin{Lemma} 
\label{hgod}
Assume $V$ satisfies \eqref{short_range33} and \eqref{long_range33} for some $m \in \N$, and $r,q \in [0,+\infty)$ not both zero. Then for any open interval $I \subset \mu(H)$ and any $\eta \in C^{\infty}_c(\R)$ with supp$(\eta) \subset I$, $PE_{I}(H) $ and $P^{\perp}\eta(H) $ $\in \mathcal{C}^1(A)$. 
\end{Lemma}
\begin{proof}
$I \subset \mu(H)$ so the Mourre estimate holds for $H$ and $A$ on $I$. In particular the point spectrum of $H$ is finite on $I$, see \cite[Corollary 7.2.11]{ABG}. Therefore $PE_{I}(H)$ is a finite rank operator. By \cite[Theorem 1.5]{Ma2}, the eigenfunctions of $H$, if any, belong to the domain of $A$. We may therefore apply Proposition \ref{rank one} to get $PE_{I}(H)  \in \mathcal{C}^1(A)$. As for $P^{\perp}\eta (H)$ it is equal to $\eta(H) - P E_{I}(H)\eta(H)$, and so belongs to 
$\mathcal{C}^1(A)$.
\qed
\end{proof}

\begin{Lemma}
\label{Lemma_compact_short_range}
Assume $V$ satisfies \eqref{short_range33} for some $q>r=2$, $m \in \N$. Let $\eta \in C_c^{\infty}(\R)$ be supported on an open interval $I$. Then $\forall \ 0 \leqslant p < q/4$, $M \geqslant m$,
\begin{equation}
\label{eqn:compact Difference1}
(\eta(H)-\eta(\Delta)) w_M ^{2p,1}(A), 
\end{equation}
\begin{equation}
\label{eqn:compact Difference2}
\quad w_M^{2p,1}(A) (\eta(H)-\eta(\Delta)) w_M^{2p,1}(A)
\end{equation}
are compact operators.
\end{Lemma}
\begin{proof}
We prove the first one and leave the second one for the reader. By Proposition \ref{goddam}, $\Delta \in \mathcal{C}^1(w_M^{2p,1}(A))$, since $w_M^{2p,1}(x) \in \mathcal{S}^{\epsilon}(\R)$, $\forall \epsilon >0$, and $\Delta \in \mathcal{C}^1(A)$. Thus $[\Delta,w_M^{2p,1}(A)]_{\circ} \in \mathscr{B}(\mathscr{H})$. 
By the Helffer-Sj{\"o}strand formula and the resolvent identity, $(\eta(H)-\eta(\Delta)) w_M^{2p,1}(A) $ equals
\begin{align*}
&\frac{\i}{2\pi} \int_{\C} \frac{\partial \tilde{\eta}}{\partial \overline{z}} (z-H)^{-1}(V+W)(z-\Delta)^{-1} w_M^{2p,1}(A) \dz \\ 
&= \frac{\i}{2\pi} \int _{\C} \frac{\partial \tilde{\eta}}{\partial \overline{z}} (z-H)^{-1}(V+W) w_M ^{2p,1}(A) (z-\Delta)^{-1} \dz \\
& \quad + \frac{\i}{2\pi} \int _{\C} \frac{\partial \tilde{\eta}}{\partial \overline{z}} (z-H)^{-1}(V+W) [(z-\Delta)^{-1},w_M ^{2p,1}(A) ]_{\circ}  \dz 
\end{align*}
\begin{align*}
&= \frac{\i}{2\pi} \int _{\C} \frac{\partial \tilde{\eta}}{\partial \overline{z}} (z-H)^{-1} \left( \underbracket{V w_M ^{2p,1} (N) }_{\text{compact by \eqref{short_range33}}} + \underbracket{W w_M ^{2p,1} (N) }_{\text{compact}} \right) \underbracket{ w_M ^{-2p,-1}(N) w_M ^{2p,1}(A) }_{\text{bounded by Lemma \ref{LogAtoN_general}}} (z-\Delta)^{-1} \dz \\
& \quad + \frac{\i}{2\pi} \int _{\C} \frac{\partial \tilde{\eta}}{\partial \overline{z}} (z-H)^{-1}\underbracket{(V+W)}_{\text{compact}} (z-\Delta)^{-1}[\Delta,w_M ^{2p,1}(A) ]_{\circ}(z-\Delta)^{-1}  \dz.
\end{align*}
The integrands of the last two integrals are compact operators. With the support of $\eta$ bounded, the integrals converge in norm, and so the compactness is preserved in the limit. 
\qed
\end{proof}

\begin{Lemma}
\label{Lemma_forProof}
Let $R>1$ and recall that $w_{M} ^{\alpha,\beta}  \left( A/R \right)$ is given by \eqref{weightsW2}. Consider 
$$\tilde{K_0} = \eD [V, \i A]_{\circ} \eD + (\eH -\eD) [H, \i A]_{\circ} \eD + \eH [H, \i A ]_{\circ}(\eH -\eD),$$
and
$$K_0 := \tilde{K_0} + \eD [W, \i A]_{\circ} \eD  = \tilde{K_0} + \eD K_W \eD + \eD B_W \eD,$$
where $\eta \in C_c ^{\infty} (\R)$ is supported on an open interval $I \Subset [0,4d] \setminus \boldsymbol{\Theta}(\Delta)$. Assume $V$ satisfies assumptions \eqref{short_range33} and \eqref{long_range33} for some $q > r = 2$, $m \in \N$. Then for all $ 1/2 < p < q/4$, $M \geqslant m$, $w_M ^ {2p, 1} \left( A/R \right) \tilde{K_0} w_M ^{2p, 1} \left( A/R \right)$ is a compact operator whose norm is uniformly bounded w.r.t.\ $R$. Furthermore, if $I \subset \mu(H)$, then further shrinking the size of the interval $I$ also allows \\ $w_M ^ {2p, 1} \left( A/R \right) K_0 w_M ^{2p, 1} \left( A/R \right)$ to be a compact operator whose norm is uniformly bounded w.r.t.\ $R$.
\end{Lemma}
\begin{proof}
Write $w_M ^ {2p, 1} \left( A/R \right) \tilde{K_0} w_M ^{2p, 1} \left( A/R \right)$ as
\begin{equation*}
\underbracket{w_M ^ {2p, 1}  \left( \frac{A}{R} \right) w_M ^ {-2p,-1} (A)}_{\text{bounded uniformly in } R}  \tilde{K}  \underbracket{ w_M  ^{-2p,-1} (A) w_M ^{2p,1}  \left( \frac{A}{R} \right)}_{\text{bounded uniformly in } R}, \quad \tilde{K}  :=  w_M  ^{2p,1}  (A)   \tilde{K_0}  w_M ^{2p,1} (A).
\end{equation*}
Thus we want to show that $\tilde{K} $ is a compact operator. $\tilde{K} $ is the sum of the following 3 terms : 
\begin{equation*}
\tilde{K_1}  := w_M  ^{2p,1} (A)  \eD [V,\i A]_{\circ} \eD  w_M  ^{2p,1} (A) ,
\end{equation*}
\begin{equation*}
\tilde{K_2}  := w_M ^{2p,1} (A) (\eta(H)-\eta(\Delta)) [H,\i A]_{\circ} \eD w_M  ^{2p,1} (A),  
\end{equation*}
\begin{equation*}
\tilde{K_3}  := w_M ^{2p,1} (A)  \eH [H,\i A]_{\circ} (\eta(H)-\eta(\Delta))  w_M  ^{2p,1} (A).   
\end{equation*}
Each of these terms are compact, as explained below. \\
\noindent $\bullet$ For $\tilde{K_1}$ : we want to commute $w_M  ^{2p,1} (A)$ with $\eD$. Since $\eD \in \mathcal{C}^1(A)$ and $w_M  ^{2p,1} (x) \in \mathcal{S}^{\epsilon}(\R)$, $\forall \epsilon > 0$, $[w_M  ^{2p,1} (A), \eD]_{\circ} \in \mathscr{B}(\mathscr{H})$ by Proposition \ref{goddam}. Also, $w_M  ^{2p,1} (A) w_M ^{-2p,-1} (N) \in \mathscr{B}(\mathscr{H})$ by Lemma \ref{LogAtoN_general}, and $w_M ^{2p,1} (N) [V, \i A]_{\circ} w_M  ^{2p,1}(N)$ is a compact operator, by the assumption \eqref{long_range33} ($1/2 < p < q/4$). Thus $\tilde{K_1}$ is a compact operator. 

\noindent $\bullet$ For $\tilde{K_2}$ : Write it as 
$$w_M  ^{2p,1} (A) (\eta(H)-\eta(\Delta)) [\Delta,\i A]_{\circ} \eD w_M  ^{2p,1} (A) + w_M ^{2p,1} (A) (\eta(H)-\eta(\Delta)) [V,\i A]_{\circ} \eD w_M ^{2p,1} (A).$$  
For the first term commute $[\Delta,\i A]_{\circ} \eD$ with $w_M ^{2p,1} (A)$ and then apply \eqref{eqn:compact Difference2}. For the second term commute $\eD$ with $w_M ^{2p,1} (A)$ and then apply \eqref{eqn:compact Difference1} and assumption \eqref{short_range33}. We leave the details to the reader.

\noindent $\bullet$ For $\tilde{K_3}$ : same idea as for $\tilde{K_2}$.

Now the second part of the statement where we assume $I \subset \mu(H)$. Recall $K_W$ and $B_W$ are given by \eqref{KW} and \eqref{BW}. It is enough to show that 
\begin{equation*}
\tilde{K_4}  := w_M ^{2p,1} (A) \eD K_W \eD w_M  ^{2p,1} (A)
\end{equation*}
and
\begin{equation*}
\tilde{K_5}  := w_M ^{2p,1} (A) \eD B_W \eD  w_M  ^{2p,1} (A) 
\end{equation*}
are compact. For $\tilde{K_4}$, commute $w_M ^{2p,1} (A)$ with $\eD$ and use the fact that $w_M ^{2p,1} (A) K_W w_M ^{2p,1} (A)$ is compact. For $\tilde{K_5}$, 
it is enough to prove that $w_M ^{2p,1} (A) \eD U_i \tilde{W} (S_i^* -S_i) \eD  w_M  ^{2p,1} (A)$ is compact. By Lemma \ref{ImportantLemma} and Proposition \ref{prop45}, $\eD \tilde{W} \eD = 0$ if the the support of $\eta$ is sufficiently small. So it is enough to prove that $w_M ^{2p,1} (A) [\eta(\Delta), U_i]_{\circ} \tilde{W} (S_i^*-S_i) \eD w_M ^{2p,1} (A)$ is compact. This can be gleaned from the following grouping
$$\underbracket{w_M ^{2p,1} (A) w_M ^{-2p,-1} (N)}_{\text{bounded by Lemma \ref{LogAtoN_general}}} \underbracket{w_M ^{2p,1} (N) [\eta(\Delta), U_i]_{\circ} w_M ^{2p,1} (N) }_{\text{compact}} \underbracket{w_M ^{-2p,-1} (N)  \tilde{W} (S_i^*-S_i) \eD w_M ^{2p,1} (A)}_{\text{bounded by Lemma \ref{LogAtoN_general}}}.$$
\qed
\end{proof}

\section{Proof of Theorem \ref{lapy3}}
\label{PROOF}

We are now ready to prove the \emph{projected weighted Mourre estimate} \eqref{general_weighted}, which in turn will imply the LAP \eqref{general_LAP}. The proof makes use of almost analytic extensions of $C^{\infty}(\R)$ bounded functions and the class of functions $\mathcal{S}^{\rho}(\R)$ with $\rho = 0$, see Appendix \ref{Helffer-Appendix}. We also mention that the proof is essentially the same as that of \cite[Theorem 4.15]{GJ2} (see also \cite[Theorem 5.4]{Ma1}), but we display it in detail for the reader's convenience. Please note that throughout all this Section we work with the Hamiltonian $H$, but the same developments work for $H_2$.



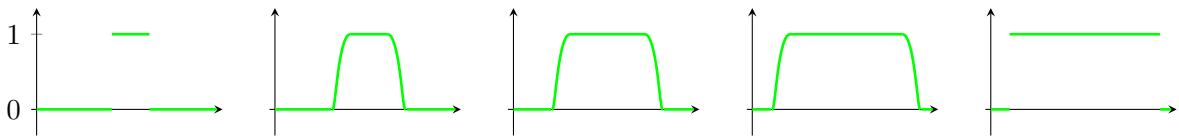
\begin{figure}[h]
\begin{tikzpicture}[domain=0.01:2.5]
\begin{axis}
[
clip = true, 
clip mode=individual, 
axis x line = middle, 
axis y line = middle, 
xticklabels={},
xtick style ={draw=none},
ylabel style={at=(current axis.above origin), anchor=south}, 
y=1cm,
x=1cm,
ytick={-3,-2,-1,0,1,2,3},
enlarge y limits={rel=0.03}, 
enlarge x limits={rel=0.03}, 
ymin = -0.3, 
ymax = 1.3, 
after end axis/.code={\path (axis cs:0,0) node [anchor=east,xshift=-0.075cm] {0};}]

\addplot[color=green,line width = 1.0 pt, samples=400,domain=1.00:1.5] ({x},{1});
\addplot[color=green,line width = 1.0 pt, samples=400,domain=0:1] ({x},{0});
\addplot[color=green,line width = 1.0 pt, samples=400,domain=1.5:2.4] ({x},{0});
\end{axis}
\end{tikzpicture}
\hspace{0.1cm}
\begin{tikzpicture}[domain=0.01:2.5]
\begin{axis}
[
clip = true, 
clip mode=individual, 
axis x line = middle, 
axis y line = middle, 
xticklabels={},
xtick style ={draw=none},
yticklabels={},
ytick style ={draw=none},
y=1cm,
x=1cm,
enlarge y limits={rel=0.03}, 
enlarge x limits={rel=0.03}, 
ymin = -0.3, 
ymax = 1.3, 
after end axis/.code={\path (axis cs:0,0) ;}]

\addplot[color=green,line width = 1.0 pt, samples=200,domain=0.76:1] ({x},{2.714*exp(-1/(1-16*(x-1)*(x-1)))});
\addplot[color=green,line width = 1.0 pt, samples=200,domain=1.5:1.74] ({x},{2.714*exp(-1/(1-16*(x-1.5)*(x-1.5)))});
\addplot[color=green,line width = 1.0 pt, samples=400,domain=1.00:1.5] ({x},{1});
\addplot[color=green,line width = 1.0 pt, samples=400,domain=0:0.75] ({x},{0});
\addplot[color=green,line width = 1.0 pt, samples=400,domain=1.75:2.4] ({x},{0});
\end{axis}
\end{tikzpicture}
\hspace{0.1cm}
\begin{tikzpicture}[domain=0.01:2.5]
\begin{axis}
[
clip = true, 
clip mode=individual, 
axis x line = middle, 
axis y line = middle, 
xticklabels={},
xtick style ={draw=none},
yticklabels={},
ytick style ={draw=none},
y=1cm,
x=1cm,
enlarge y limits={rel=0.03}, 
enlarge x limits={rel=0.03}, 
ymin = -0.3, 
ymax = 1.3, 
after end axis/.code={\path (axis cs:0,0) ;}]

\addplot[color=green,line width = 1.0 pt, samples=400,domain=0.51:0.75] ({x},{2.714*exp(-1/(1-16*(x-0.75)*(x-0.75)))});
\addplot[color=green,line width = 1.0 pt, samples=400,domain=1.75:1.99] ({x},{2.714*exp(-1/(1-16*(x-1.75)*(x-1.75)))});
\addplot[color=green,line width = 1.0 pt, samples=400,domain=0.75:1.75] ({x},{1});
\addplot[color=green,line width = 1.0 pt, samples=400,domain=0:0.5] ({x},{0});
\addplot[color=green,line width = 1.0 pt, samples=400,domain=2:2.4] ({x},{0});
\end{axis}
\end{tikzpicture}
\hspace{0.1cm}
\begin{tikzpicture}[domain=0.01:2.5]
\begin{axis}
[
clip = true, 
clip mode=individual, 
axis x line = middle, 
axis y line = middle, 
xticklabels={},
xtick style ={draw=none},
yticklabels={},
ytick style ={draw=none},
y=1cm,
x=1cm,
enlarge y limits={rel=0.03}, 
enlarge x limits={rel=0.03}, 
ymin = -0.3, 
ymax = 1.3, 
after end axis/.code={\path (axis cs:0,0) ;}]

\addplot[color=green,line width = 1.0 pt, samples=400,domain=0.26:0.5] ({x},{2.714*exp(-1/(1-16*(x-0.5)*(x-0.5)))});
\addplot[color=green,line width = 1.0 pt, samples=400,domain=2:2.24] ({x},{2.714*exp(-1/(1-16*(x-2)*(x-2)))});
\addplot[color=green,line width = 1.0 pt, samples=400,domain=0.5:2] ({x},{1});
\addplot[color=green,line width = 1.0 pt, samples=400,domain=0:0.25] ({x},{0});
\addplot[color=green,line width = 1.0 pt, samples=400,domain=2.25:2.4] ({x},{0});
\end{axis}
\end{tikzpicture}
\hspace{0.1cm}
\begin{tikzpicture}[domain=0.01:2.5]
\begin{axis}
[
clip = true, 
clip mode=individual, 
axis x line = middle, 
axis y line = middle, 
xticklabels={},
xtick style ={draw=none},
yticklabels={},
ytick style ={draw=none},
y=1cm,
x=1cm,
ytick={-3,-2,-1,0,1,2,3},
enlarge y limits={rel=0.03}, 
enlarge x limits={rel=0.03}, 
ymin = -0.3, 
ymax = 1.3, 
after end axis/.code={\path (axis cs:0,0) ;}]
\addplot[color=green,line width = 1.0 pt, samples=400,domain=0.25:2.25] ({x},{1});
\addplot[color=green,line width = 1.0 pt, samples=400,domain=0:0.25] ({x},{0});
\addplot[color=green,line width = 1.0 pt, samples=400,domain=2.25:2.4] ({x},{0});
\end{axis}
\end{tikzpicture}
\caption{From left to right : $E_J(x)$, $\theta(x)$, $\eta(x)$, $\chi(x)$, $E_I(x)$.}
\label{fig:1}
\end{figure}


\noindent \textit{Proof of Theorem \ref{lapy3}.}
$E \in \mu(H)$ so let $I \subset \mu(H)$ be an open interval containing $E$ such that the Mourre estimate holds over $I$. By Lemma \ref{hgod}, $PE_{I}(H) $ and $P^{\perp}\eta(H) $ are of class $\mathcal{C}^1(A)$, for any $\eta \in C^{\infty}_c(\R)$ with supp$(\eta) \subset I$. Let $\theta,\eta,\chi \in C^{\infty}_c (\R)$ be bump functions such that $\eta \theta = \theta$, $\chi \eta = \eta$ and supp$(\chi) \subset I$, see Figure \ref{fig:1}. Initially $J$, $\theta$ and $\eta$ are chosen and fixed so that one may apply Lemma \ref{ImportantLemma} and Proposition \ref{prop45}. As for $\chi$ and $I$ they are chosen to fulfill the conditions drawn in Figure  \ref{fig:1}. Later in the proof we will shrink the interval $I$ around $E$ (and thereby possibly $J$ as well).

We aim to derive \eqref{general_weighted} on $I$ for $B= \varphi(A/R)$, $C = \sqrt{\varphi'(A/R)}$, some $R>1$ and $\varphi$ given by 
\begin{equation}
\label{new_varphi_again}
\varphi : \R \mapsto \R, \quad \varphi : t \mapsto \int_{-\infty} ^t \frac{dx}{ \langle x \rangle w_M ^{2p, 1}(x) }, \quad t \in \R, \quad p > 1/2.
\end{equation}
Note that $\varphi \in \mathcal{S}^{0}(\R)$, so that $\varphi(A/R) \in \mathscr{B}(\mathscr{H})$ for all $R>1$. Please note that it is sufficient to obtain the LAP \eqref{LAP_this_article2} for all $1/2 < p \leq p_0$, for some $p_0 >1/2$, and only $M=m$, because then the extension to all $p>1/2$ and $M \geq m$ is automatic. Consider the bounded operator 
\begin{align*}
F & := \Pp \tH [H, \i \varphi(A/R)]_{\circ} \tH \Pp \\
&= \const \frac{1}{R} \int_{\C} \pp \Pp \tH (z-A/R)^{-1} [H, \i A]_{\circ} (z-A/R)^{-1} \tH \Pp \dz.
\end{align*}
To save room, henceforth we drop the element $\dz$. By Lemma \ref{hgod} $\Pp \eH \in \mathcal{C}^1(A)$, so 
\begin{equation*}
[\Pp \eH, (z-A/R)^{-1}]_{\circ} = (z-A/R)^{-1} [\Pp \eH, A/R]_{\circ} (z-A/R)^{-1}.
\end{equation*} 
Next to $\Pp \tH$ we introduce $\Pp\eH$ and commute it with $(z-A/R)^{-1}$. We get :
\begin{align*}
F &= \const \frac{1}{R} \int_{\C} \pp \Pp \tH \A \Pp \underbracket{ \eH [H, \i A]_{\circ}  \eH }\Pp \A \tH \Pp \\
&+ \Pp \tH \JapA ^{-\frac{1}{2}} w_M ^{-p,-\frac{1}{2}} \left( \frac{A}{R} \right) \cdot \frac{B_1}{R^2} \cdot w_M ^{-p,-\frac{1}{2}} \left( \frac{A}{R} \right) \JapA ^{-\frac{1}{2}}   \tH \Pp,
\end{align*}
where $B_1 \in \mathscr{B}(\mathscr{H})$ is uniformly bounded in $R$. We refer to \cite[Proof of Theorem 5.4]{Ma1} for extra details. Next to each $\eta(H)$ we insert $\chi(H)$, and then decompose $\eH [H, \i A ]_{\circ} \eH$ as follows :
$$\eH [H, \i A]_{\circ} \eH = \eD [\Delta, \i A]_{\circ} \eD  + K_0,$$
where
\begin{equation*}
K_0 = \eD [V+W, \i A]_{\circ} \eD + (\eH -\eD) [H, \i A]_{\circ} \eD + \eH [H, \i A ]_{\circ}(\eH -\eD).
\end{equation*}
Note that $K_0$ is the compact operator that appears in Lemma \ref{Lemma_forProof}. We also let 
$$ M_0 := \Pp \chi(H) \eD [\Delta, \i A]_{\circ} \eD \chi(H) \Pp.$$
Thus :
\begin{align}
F &= \const \frac{1}{R} \int_{\C} \pp \Pp \tH \A M_0 \A \tH \Pp \nonumber \\
\label{F--2}
&+ \const \frac{1}{R} \int_{\C} \pp \Pp \tH \A \Pp \chi(H) K_0 \chi(H) \Pp \A \tH \Pp \\
&+ \Pp \tH \JapA ^{-\frac{1}{2}} w_M ^{-p,-\frac{1}{2}} \left( \frac{A}{R} \right) \cdot \frac{B_1}{R^2} \cdot w_M ^{-p,-\frac{1}{2}} \left( \frac{A}{R} \right) \JapA ^{-\frac{1}{2}}   \tH \Pp. \nonumber
\end{align}
We write the second term in \eqref{F--2} as
$$R^{-1} \langle A/R \rangle ^{-1/2} w_M^{-p,-1/2}(A/R)  \mathcal{K} w_M^{-p,-1/2}(A/R) \langle A/R \rangle ^{-1/2}$$
where 
\begin{equation*}
\mathcal{K} := \const \int_{\C} \pp \JapA ^{\frac{1}{2}} w_M ^{p,\frac{1}{2}} \left( \frac{A}{R} \right) \A \Pp \chi(H) K_0 \chi(H) \Pp \A w_M ^{p,\frac{1}{2}}   \left( \frac{A}{R} \right)  \JapA ^{\frac{1}{2}}.
\end{equation*}

\noindent \underline{{\footnotesize CLAIM.}}

$\mathcal{K}$ is a compact operator for $1/2 < p < q/4$, where $q$ is the exponent that appears in \eqref{short_range33} and \eqref{long_range33} -- Hypothesis ($\mathbf{H}$). Also the norm of $\mathcal{K}$ goes to zero as the support of $\chi$ gets tighter around $E$, and this uniformly in $R$. \\
\noindent {\footnotesize PROOF OF CLAIM.}  Write $\mathcal{K}$ as 

\begin{align*}
\mathcal{K} =  \const \int_{\C} \pp  & \underbracket{\JapA ^{\frac{1}{2}} w_M ^{-p,-\frac{1}{2}}  \left( \frac{A}{R} \right) \A }_{\text{bounded by Lemma \ref{infamousLemma3}}} \bigtimes \\
& w_M ^{2p,1} \left( \frac{A}{R} \right)  \Pp \chi(H) K_0 \chi(H) \Pp \underbracket{ \A w_M ^{p,\frac{1}{2}}   \left( \frac{A}{R} \right)  \JapA ^{\frac{1}{2}} }_{\text{bounded by Lemma \ref{infamousLemma3}}}.
\end{align*} 
Now we want to commute $w_M ^{2p,1}  \left( \frac{A}{R} \right)$ with $ \Pp \chi(H)$.  Since $\Pp \chi(H) \in \mathcal{C}^1(A)$ and $w_M ^{2p,1} (x) \in \mathcal{S}^{\epsilon}(\R)$, $\forall \epsilon > 0$, 
$$w_M ^{2p,1} \left( \frac{A}{R} \right) \Big [w_M ^{2p,1} \left(\frac{A}{R}\right), \Pp \chi(H) \Big]_{\circ} \in \mathscr{B}(\mathscr{H})$$
by Proposition \ref{goddam}. So $\mathcal{K} = \mathcal{K} _1 + \mathcal{K} _2$, where

\begin{align*}
\mathcal{K}_1 =  \const \int_{\C} \pp  & \underbracket{\JapA ^{\frac{1}{2}} w_M ^{-p,-\frac{1}{2}}  \left( \frac{A}{R} \right) \A }_{\text{bounded by Lemma \ref{infamousLemma3}}} \bigtimes \\
&   \Pp \chi(H) w_M ^{2p,1}  \left( \frac{A}{R} \right) K_0 \chi(H) \Pp \underbracket{ \A w_M ^{p,\frac{1}{2}}   \left( \frac{A}{R} \right)  \JapA ^{\frac{1}{2}} }_{\text{bounded by Lemma \ref{infamousLemma3}}} 
\end{align*}
and
\begin{align*}
\mathcal{K}_2 =  &  \const \int_{\C} \pp  \underbracket{\JapA ^{\frac{1}{2}} w_M ^{-3p,-\frac{3}{2}}  \left( \frac{A}{R} \right) \A }_{\text{bounded by Lemma \ref{infamousLemma3}}} \bigtimes \\
& \underbracket{ w_M ^{2p,1}  \left( \frac{A}{R} \right)  \Big [ w_M ^{2p,1} \left( \frac{A}{R} \right) , \Pp \chi(H) \Big ] _{\circ} }_{\text{bounded uniformly in } R \text{ by Proposition \ref{goddam}}} K_0 \chi(H) \Pp \underbracket{ \A w_M ^{p,\frac{1}{2}}   \left( \frac{A}{R} \right)  \JapA ^{\frac{1}{2}} }_{\text{bounded by Lemma \ref{infamousLemma3}}}. 
\end{align*}
By Lemma \ref{infamousLemma3}, there is $C>0$ such that
$$\Big \| \JapA ^{\frac{1}{2}} w_M ^{-3p,-\frac{3}{2}}  \left( \frac{A}{R} \right) \A \Big \| \leqslant C w_M ^{-3p,-\frac{3}{2}} (x) \langle x \rangle ^{\frac{1}{2}} |y|^{-1}, $$
and 
$$\Big \| \JapA ^{\frac{1}{2}} w_M ^{p,\frac{1}{2}}  \left( \frac{A}{R} \right) \A \Big \| \leqslant C w_M ^{p,\frac{1}{2}} (x) \langle x \rangle ^{\frac{1}{2}} |y|^{-1}.$$
Now apply Lemma \ref{infamousLemma3} and \eqref{continuationFormula} with $\rho = 0$ and $\ell=2$. Note that $\| \mathcal{K}_2\|$ is bounded above by a multiple of $\| K_0 \chi(H) \Pp \|$ times
$$\int_{\overline{\Omega}} \underbracket{\langle x \rangle ^{\rho - 1 - \ell} |y|^{\ell}} \cdot \underbracket{w_M ^{-3p,-\frac{3}{2}} (x) \langle x \rangle ^{\frac{1}{2}} |y|^{-1}} \cdot  \underbracket{w_M ^{p,\frac{1}{2}} (x) \langle x \rangle ^{\frac{1}{2}} |y|^{-1}} dxdy = \int_{\overline{\Omega}} \frac{dx dy}{\langle x \rangle ^2 w_M ^{2p, 1}(x)} < \infty,$$ 
thanks to $p>1/2$. Because $K_0$ is compact, it follows that $\mathcal{K}_2$ is a compact operator and $\| \mathcal{K}_2 \|$ goes to zero as the support of $\chi$ gets tighter around $E$. We repeat a similar operation for $\mathcal{K}_1$ : write $w_M ^{p, \frac{1}{2}}  \left( \frac{A}{R} \right) = w_M ^{2p,1}  \left( \frac{A}{R} \right) w_M ^{-p,-\frac{1}{2}}  \left( \frac{A}{R} \right)$ and commute $w_M ^{2p,1}  \left( \frac{A}{R} \right)$ with $\chi (H) \Pp$. Doing this gives
\begin{align*}
\mathcal{K}_1 = & \const \int_{\C} \pp  \underbracket{\JapA ^{\frac{1}{2}} w_M ^{-p,-\frac{1}{2}}  \left( \frac{A}{R} \right) \A }_{\text{bounded by Lemma \ref{infamousLemma3}}} \bigtimes \\
&   \Pp \chi(H) \underbracket{ w_M ^{2p,1} \left( \frac{A}{R} \right) K_0 w_M ^{2p,1}   \left( \frac{A}{R} \right) }_{\text{compact by Lemma \ref{Lemma_forProof} }} \chi(H) \Pp \underbracket{ \A w_M ^{-p,-\frac{1}{2}}   \left( \frac{A}{R} \right)  \JapA ^{\frac{1}{2}} }_{\text{bounded by Lemma \ref{infamousLemma3}}} \ + \  \underbracket{\mathcal{K}_3}_{\text{compact}}.
\end{align*}
One shows that $\| \mathcal{K}_1 \| \leqslant c  \| \Pp \chi(H) w_M ^{2p,1}  \left( A/R \right) K_0  w_M ^{2p,1} \left( A/R \right)\|$ for some constant $c>0$ and goes to zero as the support of $\chi$ gets tighter around $E$, notably because $w_M ^{2p,1}  \left( A/R \right) K_0 w_M ^{2p,1} \left( A/R \right)$ is compact by Lemma \ref{Lemma_forProof}. 
As for $\mathcal{K}_3$, it is compact and satisfies $\| \mathcal{K}_3 \| \leqslant c \| \Pp \chi(H) w_M ^{2p,1}  \left( A/R \right) K_0  \|$ for some constant $c>0$. All in all, we see that $\mathcal{K}$ is a compact operator such that $\| \mathcal{K}\|$ goes to zero as the support of $\chi$ gets tighter around $E$, uniformly in $R$. This proves the claim. \\

We now proceed with the proof of the Theorem. Thanks to the claim we have :
\begin{align*}
F &= \const \frac{1}{R} \int_{\C} \pp \Pp \tH \A M_0 \A \tH \Pp  \\
&+  \Pp \tH \JapA ^{-\frac{1}{2}} w_M ^{-p,-\frac{1}{2}}  \left( \frac{A}{R} \right) \left( \frac{B_1}{R^2} + \frac{\mathcal{K}}{R} \right) w_M^{-p,-\frac{1}{2}}  \left( \frac{A}{R} \right) \JapA ^{-\frac{1}{2}}  \tH \Pp.
\end{align*}
The next thing to do is to commute $\A$ with $M_0$ :
\begin{align*}
F &= \const \frac{1}{R} \int_{\C} \pp \Pp \tH (z-A/R)^{-2} M_0 \tH \Pp  \\
&+ \const \frac{1}{R} \int_{\C} \pp \Pp \tH (z-A/R)^{-1} [M_0 , \A]_{\circ} \tH \Pp \\
&+ \Pp \tH \JapA ^{-\frac{1}{2}} w_M^{-p,-\frac{1}{2}}  \left( \frac{A}{R} \right)  \left( \frac{B_1}{R^2} + \frac{\mathcal{K}}{R} \right) w_M^{-p,-\frac{1}{2}} \left( \frac{A}{R} \right) \JapA ^{-\frac{1}{2}}  \tH \Pp.
\end{align*}
We apply \eqref{derivative} to the first integral (which converges in norm), while for the second integral we use the fact that $M_0 \in \mathcal{C}^1(A)$ to conclude that there exists a $B_2 \in \mathscr{B}(\mathscr{H})$ whose norm is uniformly bounded in $R$ such that 
\begin{align*}
F &= R^{-1} \Pp \tH \varphi'(A/R) M_0 \tH \Pp \\
&+  \Pp \tH \JapA ^{-\frac{1}{2}} w_M^{-p,-\frac{1}{2}}  \left( \frac{A}{R} \right) \left( \frac{B_2}{R^2} + \frac{\mathcal{K}}{R} \right) w_M^{-p,-\frac{1}{2}} \left( \frac{A}{R} \right)  \JapA ^{-\frac{1}{2}}  \tH \Pp.
\end{align*}
Now $\varphi'(A/R) = \langle A/R \rangle ^{-1} w_M^{-2p,-1}  (A/R) $. By Proposition \ref{goddam},
\begin{equation*}
[\langle A/R \rangle ^{-\frac{1}{2}} w_M^{-p,-\frac{1}{2}}  (A/R), M_0 ]_{\circ} \langle A/R \rangle ^{\frac{1}{2}} w_M^{p,\frac{1}{2}}  (A/R) = R^{-1} B'
\end{equation*} 
for some $B' \in \mathscr{B}(\mathscr{H})$ whose norm is uniformly bounded in $R$. Thus there is $B_3 \in \mathscr{B}(\mathscr{H})$ with norm uniform in $R$ such that
\begin{align*}
F &= R^{-1} \Pp \tH \JapA ^{-\frac{1}{2}} w_M^{-p,-\frac{1}{2}} \left( \frac{A}{R} \right)  M_0 w_M^{-p,-\frac{1}{2}} \left( \frac{A}{R} \right)  \JapA ^{-\frac{1}{2}}   \tH \Pp \\
&+  \Pp \tH \JapA ^{-\frac{1}{2}} w_M^{-p,-\frac{1}{2}} \left( \frac{A}{R} \right)  \left( \frac{B_3}{R^2} + \frac{\mathcal{K}}{R} \right) w_M^{-p,-\frac{1}{2}} \left( \frac{A}{R} \right)  \JapA ^{-\frac{1}{2}}   \tH \Pp \\
& \geqslant \gamma R^{-1} \Pp \tH  \JapA ^{-\frac{1}{2}} w_M^{-p,-\frac{1}{2}}\left( \frac{A}{R} \right)  \Pp \chi(H) \eta^2(\Delta) \chi(H) \Pp w_M^{-p,-\frac{1}{2}}\left( \frac{A}{R} \right)  \JapA ^{-\frac{1}{2}}  \tH \Pp  \\
&+  \Pp \tH  \JapA ^{-\frac{1}{2}}w_M^{-p,-\frac{1}{2}} \left( \frac{A}{R} \right)  \left( \frac{B_3}{R^2} + \frac{\mathcal{K}}{R} \right) w_M^{-p,-\frac{1}{2}} \left( \frac{A}{R} \right)  \JapA ^{-\frac{1}{2}} \tH \Pp
\end{align*}
where $\gamma>0$ comes from applying the Mourre estimate for $\Delta$ and $A$. Let
\begin{equation*}
\tilde{\mathcal{K}} := \gamma \Pp \chi(H)(\eta^2(\Delta)-\eta^2(H))\chi(H) \Pp.
\end{equation*} 
Note that $\tilde{\mathcal{K}}$ is compact with $\|\tilde{\mathcal{K}}\|$ vanishing as the support of $\chi$ gets tighter around $E$. Thus 
\begin{align*}
F & \geqslant \gamma R^{-1} \Pp \tH \JapA ^{-\frac{1}{2}} w_M^{-p,-\frac{1}{2}} \left( \frac{A}{R} \right)  \Pp \chi(H) \eta^2(H) \chi(H) \Pp  w_M^{-p,-\frac{1}{2}}\left( \frac{A}{R} \right)  \JapA ^{-\frac{1}{2}}  \tH \Pp  \\
&+  \Pp \tH \JapA ^{-\frac{1}{2}} w_M^{-p,-\frac{1}{2}} \left( \frac{A}{R} \right)  \left( \frac{B_3}{R^2} + \frac{\mathcal{K} + \tilde{\mathcal{K}}}{R} \right) w_M^{-p,-\frac{1}{2}} \left( \frac{A}{R} \right)  \JapA ^{-\frac{1}{2}}  \tH \Pp.
\end{align*}
Finally, we commute $\Pp \chi(H) \eta^2(H) \chi(H) \Pp = \Pp \eta^2(H)  \Pp$ with $w_M^{-p,-\frac{1}{2}} (A/R)  \langle A / R \rangle ^{-\frac{1}{2}} $, and see that 
\begin{equation*}
\gamma [\Pp \eta^2(H) \Pp, w_M^{-p,-\frac{1}{2}} (A/R)  \langle A / R \rangle ^{-\frac{1}{2}} ]_{\circ} w_M^{p, \frac{1}{2}} (A/R)  \langle A / R \rangle ^{\frac{1}{2}} = R^{-1} B''
\end{equation*}
for some $B'' \in \mathscr{B}(\mathscr{H})$ whose norm is uniformly bounded in $R$. Thus there is $B_4 \in \mathscr{B}(\mathscr{H})$ with norm uniform in $R$ such that
\begin{align*}
F & \geqslant \gamma R^{-1} \Pp \tH  \JapA ^{-1} w_M^{-2p,-1}  \left( \frac{A}{R} \right) \tH \Pp  \\
&+  \Pp \tH \JapA ^{-\frac{1}{2}} w_M^{-p,-\frac{1}{2}} \left( \frac{A}{R} \right)  \left( \frac{B_4}{R^2} + \frac{\mathcal{K} + \tilde{\mathcal{K}}}{R} \right)  w_M^{-p,-\frac{1}{2}}  \left( \frac{A}{R} \right)  \JapA ^{-\frac{1}{2}}  \tH \Pp.
\end{align*}
To conclude, we shrink the support of $\chi$ to ensure that $\|\mathcal{K} + \tilde{\mathcal{K}}\| < \gamma/3$ and choose $R > 1$ so that $\| B_4 \|/R < \gamma/3$. Then $\mathcal{K} + \tilde{\mathcal{K}} \geqslant -\gamma/3$ and $B_4 /R \geqslant -\gamma/3$, so 
\begin{equation}
F = \Pp \tH \Big [H, \i \varphi \left( \frac{A}{R} \right) \Big]_{\circ} \tH \Pp \geqslant \frac{\gamma}{3R} \Pp \tH \JapA ^{-1} w_M^{-2p,-1} \left( \frac{A}{R} \right)  \tH \Pp.
\end{equation}
Let $J'$ be any open interval with $J' \Subset I$. Applying $E_{J'}(H)$ on both sides of this inequality yields the projected weighted Mourre estimate \eqref{general_weighted}, with $c=\gamma/(3R)$ and $C = \langle A/R \rangle ^{-\frac{1}{2}} w_M^{-p,-\frac{1}{2}}  \left( A/R\right)$, $1/2<p<q/4$. The proof of Theorem \ref{lapy3} is complete, as explained in Section \ref{weightedMestimate}.
\qed








\section{Appendix A. Nevanlinna functions, operator monotone functions and Loewner's theorem} 

\label{Nevanlinna}

We revisit Loewner's theorem on matrix operator monotone functions, see e.g.\ \cite{L}, \cite{Do}, \cite{Si2} and \cite{Ha}. This wonderful theorem makes a striking connection between the operator monotone functions and the Nevanlinna functions.

Let $\C_+$ (resp.\ $\C_-$) denote the complex numbers with strictly positive (resp.\ negative) imaginary part. A \emph{Nevanlinna function} (also known as Herglotz, Pick or R function) is an analytic function on $\C_+$ that maps $\C_+$ to $\overline{\C_+}$. A function $f : \C_+ \mapsto \C$ is Nevanlinna if and only if it admits a representation
\begin{equation}
\label{Nevanlinna3}
f(z) = \alpha + \beta z + \int_{\R} \left( \frac{1}{\lambda - z} - \frac{\lambda}{\lambda^2+1} \right) d \mu (\lambda), \quad z \in \C^+
\end{equation}
where $\alpha \in \R$, $\beta \geqslant 0$, and $\mu$ is a positive Borel measure on $\R$ satisfying $\int_{\R} (\lambda^2+1) ^{-1} d\mu(\lambda) < \infty$. We refer to \cite[Theorem 1 of Chapter II]{Do} for a proof of this wonderful result. The integral representation is unique. The measure $\mu$ is recovered from $f$ by the Stieltjes inversion formula
$$\mu \left( (\lambda_1, \lambda_2] \right) = \lim \limits_{\delta \downarrow 0} \lim \limits_{\epsilon \downarrow 0} \frac{1}{\pi} \int_{\lambda_1 + \delta} ^{\lambda_2 + \delta} \mathrm{Im} \left( f(\lambda + \i \epsilon) \right) d\lambda.$$
Standard examples of Nevanlinna functions given in the literature include $z^{p}$ for $0 \leqslant p \leqslant 1$, $-z^{p}$ for $-1 \leqslant p \leqslant 0$, and the logarithm $\ln(z)$ with the branch cut $(-\infty,0]$.  
\\

\noindent \textbf{Notation.} Let $(a,b)$ be an open interval (finite or infinite).  Denote $P(a,b)$ the set of Nevanlinna functions that continue analytically across $(a,b)$ into $\C_-$ and where the continuation is by reflection. \\

Functions in $P(a,b)$ are real-valued on $(a,b)$ and their measure satisfies $\mu \left( (a,b) \right) = 0$, see \cite[Lemma 2, Chapter II]{Do}. Functions in $P(a,b)$ are strictly increasing on $(a,b)$, unless they are constant. Indeed, if $f$ is not a constant function $\beta$ and $\mu$ cannot be simulaneously zero and so
$$f'(x) = \beta + \int _{\R} \frac{d\mu (\lambda)}{(\lambda-x)^2} > 0, \quad x \in (a,b).$$
Let $\mathscr{M}_n(\C)$ be the set of $n \times n$ matrices with entries in $\C$, and consider a function $f : (a,b) \mapsto \R$. 

\noindent \textbf{Definition.}
$f$ is \emph{matrix monotone of order $n$ in} $(a,b)$ if $f(T) \leqslant f(S)$ holds whenever $T,S$ in $\mathscr{M}_n(\C)$ are hermitian matrices with spectrum in $(a,b)$ and $T \leqslant S$. \\

\noindent In 1934 Karl Loewner proved the following remarkable theorem that characterizes the matrix monotone functions : 

\begin{theorem} 
\label{LoewnerOriginal} \cite{L} Let $f : (a,b) \mapsto \R$, where $(a,b)$ is a finite or infinite open interval. Then $f$ is matrix operator monotone of order $n$ in $(a,b)$ for all $n \in \N$ if and only if $f$ admits an analytic continuation that belongs to $P(a,b)$.
\end{theorem}
Loewner's theorem is a truly wonderful result and has been reproved in several different ways, see e.g.\ \cite[Theorem I of Chapter VII and Chapter IX for the converse]{Do}, \cite{Si2} or \cite{Ha} and references therein for a concise historical exposition. For the purpose of this article, we need a version of Loewner's theorem that applies to unbounded self-adjoint operators. In B.\ Simon's book \cite[Chapter 2]{Si2} it is explicitly discussed how Loewner's theorem extends to unbounded operators. We propose below yet another proof of the extension to the semi-bounded operators (which is the case for $\langle A \rangle$ and $\langle N \rangle$). For self-adjoint operators $T$, $S$ which are bounded from below, the inequality $T \leqslant S$ means $\mathrm{Dom}[|S|^{1/2}]\subset \mathrm{Dom}[|T|^{1/2}]$ and
 $\langle  \psi, T \psi \rangle \leqslant \langle \psi, S \psi \rangle$ for all $\psi \in \mathrm{Dom}[|S|^{1/2}]$.
\\

\noindent \textbf{Definition.} $f$ is \emph{operator monotone in} $(a,b) \subset \R$ if $f(T) \leqslant f(S)$ holds whenever $T,S$ (possibly unbounded) are self-adjoint operators in $\mathscr{H}$ with spectrum contained in $(a,b)$ and $T \leqslant S$.
\\ 

Assuming $f \in P(a,b)$, let $\mu$ be the measure associated to $f$, see \eqref{Nevanlinna3}, and denote the supremum of the support of $\mu$ by $\Sigma_{\mu}$. In what follows the discussion is for general self-adjoint $T,S$, but one should keep in mind that the results are applied to $T =\langle A \rangle$ and $S = \sqrt{c_d} \langle N \rangle$, cf.\  Lemma \ref{A<N}. We start with a Lemma :
\begin{Lemma} 
\label{lemmaDomains}
Let $f$ belong to $P(a,+\infty)$. Then
\begin{enumerate}
\item $\Sigma_{\mu} \leqslant a$,
\item given $T$ a self-adjoint operator with $\inf (\sigma(T))> a >0$, we have $f(T)$ bounded from below and $\mathrm{Dom}[T] \subseteq \mathrm{Dom}[|f(T)|]$ with equality if and only if $\beta >0$.
\end{enumerate}
\end{Lemma}
\begin{proof}
For (1), we start by noting that $f$ belongs to $P(a,+\infty)$ so $f$ admits an integral representation as in \eqref{Nevanlinna3}. Thus $\Sigma_{\mu} \leqslant a$ holds thanks to the Stieltjes inversion formula and the fact that $f$ is real-valued on $(a,+\infty)$. For (2), adding a constant to $f$ does not alter the assumptions of the Lemma and leads to the same conclusion, so we may assume  $f(\inf(\sigma(T))) >0$, where $\inf(\sigma(T)) > a$ (recall $f$ is non-decreasing on $(a,+\infty)$). We start by showing that $\lim f(x)/x = \beta$, as $x\to +\infty$, or equivalently
$$ \lim \limits_{x \to +\infty} \int_{-\infty} ^{\Sigma_{\mu}} \frac{1+\lambda x}{x(\lambda -x)(\lambda^2+1)} d\mu(\lambda) = 0.$$ 
We wish to exchange the order of the limit and integration. We have 
\begin{equation}
\label{absolute}
\big | (1+\lambda x)x^{-1}(\lambda -x)^{-1} \big | \leqslant 1, \quad \forall (\lambda,x) \in (-\infty,0] \times [1,+\infty) \cup [0,a] \times [a+\sqrt{a^2+1}, + \infty).
\end{equation}
We may apply the dominated convergence theorem, and the above limit follows.
 This limit implies that $f(x)/x$ is a bounded function on $(a,+\infty)$ and hence $\mathrm{Dom}[T] \subset \mathrm{Dom}[|f(T)|]$. For the reverse inclusion, $x/f(x)$ is a well defined bounded function on $(\inf(\sigma(T)),+\infty)$ iff $\beta >0$.
\qed
\end{proof}

We are now ready to prove the extension of Theorem \ref{LoewnerOriginal} to semi-bounded operators :

\begin{theorem} 
\label{Prop33}
Let $f : (a,b) \mapsto \R$, where $0 < a < b \leqslant +\infty$. Then $f$ is operator monotone in $(a,b)$ if and only if $f$ admits an analytic continuation that belongs to $P(a,b)$. 
\end{theorem}

\begin{proof} 
If $f$ is operator monotone in $(a,b)$, then in particular it is matrix operator monotone of order $n$ in $(a,b)$ for all $n \in \N$, and so $f$ admits an analytic continuation that belongs to $P(a,b)$ by Loewner's theorem. This direction is in fact the hard direction in Loewner's theroem, but the extension is trivial! 

For the converse, we suppose that $f$ admits an analytic continuation that belongs to $P(a,b)$. Consider the two separate cases $b < +\infty$ and $b = +\infty$.
If $b < +\infty$, then $f$ is matrix operator monotone of order $n$ in $(a,b)$ for all $n \in \N$ by Loewner's theorem. But, since the interval $(a,b)$ is finite, this is equivalent to being operator monotone in $(a,b)$, see e.g.\ \cite[Lemma 2.2]{BS} or \cite[Chapter 2]{Si2}. Now the case $b=+\infty$. We have $\Sigma_{\mu} \leqslant a$ by Lemma \ref{lemmaDomains}, and 
$$f(x) = \alpha + \beta x + \int _{-\infty} ^{\Sigma_{\mu}} \left( \frac{1}{\lambda - x} - \frac{\lambda}{\lambda^2+1} \right) d \mu (\lambda), \quad x \in (a,+\infty) .$$
Let
\begin{equation}
\label{sequenceFr1}
g_r (x) :=  \alpha + \int_{-r} ^{\Sigma_{\mu}}  \left( \frac{1}{\lambda - x} - \frac{\lambda}{\lambda^2+1} \right) d \mu (\lambda), \quad g(x) := f(x) - \beta x, \quad x \in (a,+\infty).
\end{equation}
Note that $\{g_r \}_{r \in \R^+}$ is of a sequence of functions of the real variable $x$ that converges pointwise to $g(x)$ as $r \to + \infty$ for all $x \in (a,+\infty)$. Moreover, since for every fixed $r$, $g'_r(x) \geqslant 0$, all functions in the sequence are increasing in the variable $x$. We need a Lemma.

\begin{Lemma} 
\label{Lem11}
For every fixed $\epsilon > 0$, the sub-sequence $\{ g_r(x) \}_{r > 1/ \epsilon}$ has the property that $g_r(x) \nearrow g(x)$ for every $x >\max(\Sigma_{\mu} + \epsilon, a)$.
\end{Lemma}
\begin{proof}
Clearly $g_r(x) \to g(x)$ as $r\to +\infty$ for all $x \in (a,+\infty)$. What needs to be shown is that the sequence $g_r(x)$ is increasing pointwise. Let $\epsilon > 0$ be given and fix $x > \max(\Sigma_{\mu} + \epsilon, a)$. Recall $a$ is assumed to be strictly positive. The integrand in \eqref{sequenceFr1} is equal to 
$$\frac{1+\lambda x}{(\lambda-x) (\lambda^2+1)}.$$
On the one hand, $x > \Sigma_{\mu}+ \epsilon$ implies that $(\lambda-x)$ is negative for all $\lambda \in \mathrm{supp} \ \mu$. On the other hand, the assumption $x > a > 0$ implies that $(1+\lambda x)$ is negative for all $\lambda < -1/x$. Thus, for all $x >\max(\Sigma_{\mu} + \epsilon, a)$ and for all $\lambda < -1/a$, the integrand in \eqref{sequenceFr1} is strictly positive. Thus, as $r$ increases above $1/a$, the value of the integral in \eqref{sequenceFr1} strictly increases. This completes the proof of the Lemma.
\qed
\end{proof}

\noindent \textit{Continuation of the proof of Theorem \ref{Prop33}.}
Now fix $T,S$ self-adjoint operators with spectrum contained in $(a,+\infty)$ and $T \leqslant S$. 
Let $\psi \in \mathscr{H}$. $T \leqslant S$ implies $\langle \psi, (\lambda - T)^{-1} \psi \rangle \leqslant \langle \psi, (\lambda - S)^{-1} \psi \rangle$ for all $\lambda \leqslant \Sigma_{\mu}$, e.g. \cite[Theorem VI.2.21]{Ka}.
We integrate over the compact $[-r,\Sigma_{\mu}]$ :
$$ \int_{-r}^{\Sigma_{\mu}} \langle \psi, (\lambda - T)^{-1} \psi \rangle \ d\mu(\lambda) \leqslant \int_{-r}^{\Sigma_{\mu}} \langle \psi, (\lambda - S)^{-1} \psi \rangle \ d\mu(\lambda), \quad \text{for every finite} \ r > |\Sigma_{\mu}|.$$
Moreover, since the integrand is norm continuous, we infer 
\begin{equation}
\label{eqn11}
\langle \psi, g_r(T) \psi \rangle \leqslant \langle \psi, g_r(S) \psi\rangle, \quad \text{for every finite} \ r > |\Sigma_{\mu}|.
\end{equation}
Next we choose in Lemma \ref{Lem11} $\varepsilon$ small enough so that $\max(\Sigma_\mu+\varepsilon, a)<\inf(\sigma(T))\leqslant \inf(\sigma(S))$. 
We obtain $g_r(x) \nearrow g(x)$ (for $r>1/\varepsilon$) as $r$ goes to infinity. 
In turn, the monotone convergence theorem for forms, 
e.g.\ \cite[Theorem VIII.3.11]{Ka}, 
ensures that $\langle  \psi, g_r(T) \psi\rangle$ converges to $\langle  \psi, g(T)\psi\rangle$ as $r \to +\infty$ for all $\psi \in \mathrm{Dom}[|g(T)|^{1/2}]$, and similarly for $S$. Taking limits in \eqref{eqn11} gives 
$\mathrm{Dom}[|g(S)|^{1/2}]\subset \mathrm{Dom}[|g(T)|^{1/2}]$ and
$$\langle \psi, g(T) \psi \rangle \leqslant \langle \psi, g(S) \psi\rangle, \quad \psi \in \mathrm{Dom}[|g(S)|^{1/2}].$$
Thus, if $\beta = 0$ we have
$$\langle \psi, f(T) \psi \rangle \leqslant \langle \psi, f(S) \psi\rangle, \quad \psi \in \mathrm{Dom}[|g(S)|^{1/2}]= \mathrm{Dom}[|f(S)|^{1/2}],$$
whereas if $\beta > 0$, we use the fact that $0 < T\leqslant S$, and $\mathrm{Dom}[S^{1/2}] = \mathrm{Dom}[|f(S)|^{1/2}]$, $\mathrm{Dom}[T^{1/2}]= \mathrm{Dom}[|f(T)|^{1/2}]$,  by Lemma \ref{lemmaDomains}, which yields
$$\langle \psi, f(T) \psi \rangle \leqslant \langle \psi, f(S) \psi\rangle, \quad \psi \in \mathrm{Dom}[S^{1/2}] = \mathrm{Dom}[|f(S)|^{1/2}].$$
This gives the result. \qed
\end{proof}


To close this Section, we have a remark about results in the literature on order relations for general self-adjoint operators. In \cite{O}, Olson introduced the \emph{spectral order} for self-adjoint operators : $T \preceq S$ if and only if $E_{(-\infty,t]}(T) \geqslant E_{(-\infty,t]}(S)$ for all $t \in \R$. Here $E_{(-\infty,t]}(T)$ and $E_{(-\infty,t]}(S)$ are the spectral resolutions of the identity for $T$ and $S$. He showed that $T^n \leqslant S^n$ for every $n \in \N$ is equivalent to $T \preceq S$, see also \cite[Proposition 5]{U}. Furthermore, it is shown in \cite{FK} that this order relation is equivalent to $f(T) \leqslant f(S)$ for any continuous monotone nondecreasing function $f$ defined on an interval which contains $\sigma(T) \cup \sigma(S)$. For the purpose of this article, we do not know if $\langle A \rangle ^n \leqslant (\sqrt{c_d} \langle N \rangle)^n $ holds $\forall n \in \N$ but if it does it would considerably simplify this article. To check this inequality by brute force seems to be unbearable.

\section{Appendix B. Polylogarithms of positive order are Nevanlinna functions}
\label{Polylog}

That logarithm with the standard branch cut is a Nevanlinna function follows from the identity $\ln(z) = \ln(r) + \i \theta$, where $z = r e^{\i \theta}$, $r>0$, and $\theta \in (-\pi,\pi)$. The integral representation of the logarithm is 
\begin{equation*}
\label{log}
\ln(z) = \int_{-\infty} ^0 \left(\frac{1}{\lambda-z} - \frac{\lambda}{\lambda^2+1} \right) d\lambda.
\end{equation*}
The composition of Nevanlinna functions produces another Nevanlinna function. So for example $\ln^p(z)$ is Nevanlinna for $0 \leqslant p \leqslant 1$. What about higher powers of the logarithm ? Certainly the square and cube of the logarithm are not Nevannlina functions. Indeed writing 
$$\ln^2(z) = \ln^2(r) - \theta^2 + \i 2 \theta \ln(r), \quad \text{and} \quad \ln^3(z) = \ln^3(r) - 3\theta^2 \ln(r) + \i \theta( 3\ln^2(r) -\theta^2)$$
reveals that these functions do not map $\C_+$ to $\overline{\C_+}$. In spite of this there are functions that are Nevanlinna and are "almost" equal to the logarithms. To motivate the idea, we note that the Stieltjes inversion formula gives

$$\lim \limits_{\delta \downarrow 0} \lim \limits_{\epsilon \downarrow 0} \frac{1}{\pi} \int_{\lambda_1 + \delta} ^{\lambda_2 + \delta} \mathrm{Im} \left( \ln^2(\lambda + \i \epsilon) \right) d\lambda = \begin{cases}
0 & \lambda_1 \geqslant 0, \\
\int_{\lambda_1} ^{\lambda_2} 2\ln(|\lambda|) d\lambda & \lambda_2 \leqslant 0 \\
\end{cases}
$$
when applied to $\ln^2(z)$ and 

$$\lim \limits_{\delta \downarrow 0} \lim \limits_{\epsilon \downarrow 0} \frac{1}{\pi} \int_{\lambda_1 + \delta} ^{\lambda_2 + \delta} \mathrm{Im} \left( \ln^3(\lambda + \i \epsilon) \right) d\lambda = \begin{cases}
0 & \lambda_1 \geqslant 0, \\
\int_{\lambda_1} ^{\lambda_2} \left( 3\ln^2(| \lambda |) - \pi^2 \right) d\lambda & \lambda_2 \leqslant 0 \\
\end{cases}
$$
when applied to $\ln^3(z)$. This suggests to calculate the Nevanlinna functions corresponding to the measures 
$\d\mu(\lambda) = \bm{1}_{\{\lambda < -1\}} \ln(-\lambda) d\lambda $ and $\d\mu(\lambda) = \bm{1}_{\{\lambda < -1\}} \ln^2(-\lambda) d\lambda$.

We introduce polylogarithms. We refer to \cite{Le} and \cite{PBM} for formulas and a detailed exposition. The polylogarithm of order $\sigma \in \C$ is defined by the power series 
 $$\mathrm{Li}_{\sigma}(z) := \sum _{k=1} ^{\infty} \frac{z^k}{k^{\sigma}}.$$
 The definition is valid for complex $|z| <1$ and is extended to the complex plane by analytic continuation. For the purpose of this article, we are interested in the polylogarithms with $\sigma > 2$, or $\sigma = 3$ if we want to simplify by taking the smallest integer above 2. The standard branch cut is $[1,+\infty)$ for $\mathrm{Li}_1(z)$ and $(1,+\infty)$ for $\mathrm{Li}_2(z)$ and $\mathrm{Li}_3(z)$. The polylogarithm of order 1 can be written in terms of a logarithm as $\mathrm{Li}_1(z) = - \ln(1-z)$. The polylogarithm of order 2 is called the \textit{dilogarithm} or \textit{Spence function} while the polylogarithm of order 3 is called the \textit{trilogarithm}. On p.\ 494 of \cite{PBM} the following integral representation is given without proof :
\begin{equation}
\label{Int_rep_PBM}
\mathrm{Li}_{\sigma+1} (z) = \frac{z}{\Gamma (\sigma+1)}  \int_{1} ^{\infty} \frac{\ln ^{\sigma} (\lambda) }{\lambda (\lambda-z)} d\lambda , 
\end{equation}
for $| \mathrm{arg} (1-z) | < \pi, \mathrm{Re}(\sigma) > -1$, or for $ z=1,  \mathrm{Re}(\sigma) > 0$. Here $\Gamma$ is the Gamma function.
Obviously \eqref{Int_rep_PBM} is equivalent to 
\begin{equation}
\label{Int_rep_PBM2}
\mathrm{Li}_{\sigma+1} (z) = - \frac{1}{\Gamma (\sigma+1)}  \int_{1} ^{\infty}  \frac{\ln ^{\sigma} (\lambda) }{\lambda (\lambda^2+1)} d\lambda + \frac{1}{\Gamma (\sigma+1)}  \int_{1} ^{\infty} \left( \frac{1}{\lambda-z} - \frac{\lambda}{ \lambda^2+1} \right) \ln ^{\sigma} (\lambda) d\lambda.
\end{equation}
This means that $\mathrm{Li}_{\sigma+1}(z)$ is a Nevanlinna function for $\mathrm{Re}(\sigma) > -1$. Although not difficult to prove, it is not clear where a proof of \eqref{Int_rep_PBM} can be found in the literature. Thus we prove it :

\begin{proposition}
\eqref{Int_rep_PBM} is true.
\end{proposition}
\begin{proof}
Let $\lambda \geqslant 1$. Writing $1/(\lambda(\lambda-z))$ as a power series in $z$ we have $1/(\lambda(\lambda-z)) = \sum_{k=0} ^{\infty} \lambda^{-k-2} z^k$ for $|z| < 1$. Then the rhs of \eqref{Int_rep_PBM} is equal to 
$$ \frac{z}{\Gamma (\sigma+1)}  \int_{1} ^{\infty} \sum_{k=0} ^{\infty} \frac{\ln ^{\sigma} (\lambda) }{\lambda ^{k+2}} z^k d\lambda =  \sum_{k=1} ^{\infty} \frac{z^k}{\Gamma (\sigma+1)}  \int_{1} ^{\infty} \frac{\ln ^{\sigma} (\lambda) }{\lambda ^{k+1}} d\lambda = \sum_{k=1} ^{\infty} \frac{z^k}{k^{\sigma+1}}, \quad \text{Re}(\sigma) >-1.$$
To evaluate the last integral the change of variable $k \ln(\lambda) = t$ was performed, followed by the definition of the Gamma function. Thus the rhs of \eqref{Int_rep_PBM} is equal to $\mathrm{Li}_{\sigma+1}(z)$ for $|z|<1$ and $\text{Re}(\sigma) >-1$. The result follows by the uniqueness of the analytic continuation.
\qed
\end{proof}
While we're at it we note that $\mathrm{Li}_{0} (z) = z/(1-z)$ is also a Nevanlinna function.

\noindent \textbf{Definition.} For $\sigma \in \C$,
\begin{equation}
\label{PHI}
\Phi_{\sigma} (z) := - \mathrm{Li}_{\sigma}(-z), \quad z \in \C \setminus (-\infty,-1].
\end{equation}
Clearly \eqref{Int_rep_PBM2} implies that the $\Phi_{\sigma}$ are Nevanlinna for $\text{Re}(\sigma) >-1$ with integral representations given by :
$$\Phi_{\sigma+1} (z) = \frac{1}{\Gamma (\sigma+1)}  \int_{1} ^{\infty}  \frac{\ln ^{\sigma} (\lambda) }{\lambda (\lambda^2+1)} d\lambda + \frac{1}{\Gamma (\sigma+1)}  \int_{-\infty} ^{-1} \left( \frac{1}{\lambda-z} - \frac{\lambda}{ \lambda^2+1} \right) \ln ^{\sigma} (-\lambda) d\lambda,$$
for $|\text{arg}(1+z) | < \pi$. Finally, the other reason we resort to polylogarithms is because they decay at the same rate as the logarithms, at least for positive integer order (this follows directly from the inversion/reflection formula \cite[(6) of Appendix A.2.7]{Le} together with $\mathrm{Li}_n(0)=0$), namely : 
 \begin{equation}
 \label{limits Li}
 \lim \limits_{x \to +\infty} \frac{\Phi_n(x)}{\ln ^ n (x)} = \frac{1}{n!}, \quad n \in \N.
 \end{equation}

\section{Appendix C. Almost analytic extensions and Helffer-Sj\"ostrand calculus}
\label{Helffer-Appendix}

We refer to \cite{D}, \cite{DG}, \cite{GJ1}, \cite{GJ2}, \cite{MS} for more details. Let $\rho \in \R$ and denote by $\mathcal{S}^{\rho}(\R)$ the class of functions $\varphi$ in $C^{\infty}(\R)$ such that 
\begin{equation}
|\varphi^{(k)}  (x)| \leqslant C_k \langle x \rangle ^{\rho-k}, \quad \text{for all} \ k \in \N.
\label{decay1}
\end{equation}

\begin{Lemma} \cite{D} and \cite{DG}
\label{DGDavies}
Let $\varphi \in \mathcal{S}^{\rho}(\R)$, $\rho \in \R$. Then for every $N \in \Z^+$ and every $c>0$, there exists a smooth function $\tilde{\varphi}_N : \C \to \C$, called an almost analytic extension of $\varphi$, satisfying:
\begin{equation*}
\tilde{\varphi}_N(x+\i 0) = \varphi(x), \quad \forall x \in \R;
\end{equation*}
\begin{equation*}
\mathrm{supp} \ (\tilde{\varphi}_N) \subset \Omega := \{x+\i y : |y| \leqslant c \langle x \rangle \};
\end{equation*}
\begin{equation*}
\tilde{\varphi}_N(x+\i y) = 0, \quad \forall y \in \R \ \mathrm{whenever} \ \varphi(x) = 0;
\end{equation*}
\begin{equation}
\label{continuationFormula}
\forall \ell \in \N \cap [0,N], \Bigg| \frac{\partial \tilde{\varphi}_N}{\partial \overline{z}}(x+\i y) \Bigg| \leqslant c_{\ell} \langle x \rangle ^{\rho -1-\ell} |y|^{\ell} \ \mathrm{for \ some \ constants} \ c_{\ell} >0.
\end{equation}
\end{Lemma}

Now let $Q$ be a self-adjoint operator.
\begin{Lemma}
Let $\rho < 0$ and $\varphi \in \mathcal{S}^{\rho}(\R)$. Then for all $k \in \N$ and $N \in \N$:
\begin{equation}
\label{derivative}
\varphi^{(k)}(Q) = \frac{\i (k!)}{2\pi} \int _{\C} \frac{\partial \tilde{\varphi}_N}{\partial \overline{z}} (z) (z-Q)^{-1-k} dz \wedge d\overline{z}
\end{equation}
where the integral exists in the norm topology. For $\rho \geqslant 0$, the following limit exists:
\begin{equation}
\varphi^{(k)}(Q)f = \lim \limits_{R \to \infty} \frac{\i (k!)}{2\pi} \int_{\C} \frac{\partial (\tilde{\varphi\theta_R})_N}{\partial \overline{z}} (z) (z-Q)^{-1-k} f dz \wedge d\overline{z}, \quad \text{for all} \  f \in \mathrm{Dom}[\langle Q \rangle ^{\rho}].
\label{eqn:lem}
\end{equation}
In particular, if $\varphi \in \mathcal{S}^{\rho}(\R)$ with $0 \leqslant \rho < k$ and $\varphi^{(k)}$ is a bounded function, then $\varphi^{(k)}(Q)$ is a bounded operator and \eqref{derivative} holds (with the integral converging in norm).
\end{Lemma}

\begin{proposition} \cite{GJ1} Let $T$ be a bounded self-adjoint operator satisfying $T \in \mathcal{C}^1(Q)$. Then :
\begin{equation}
\label{use1}
[T,(z-Q)^{-1}]_{\circ} = (z-Q)^{-1}[T,Q]_{\circ}(z-Q)^{-1},
\end{equation}
and for any $\varphi \in \mathcal{S}^{\rho}(\R)$ with $\rho < 1$, $T \in \mathcal{C}^1(\varphi(Q))$ and
\begin{equation}
\label{use2}
[T,\varphi(Q)]_{\circ} = \const \int_{\C} \frac{\partial \tilde{\varphi}_N}{\partial \overline{z}} (z-Q)^{-1}[T,Q]_{\circ}(z-Q)^{-1} \dz.
\end{equation}
\label{goddam}
\end{proposition}

\end{document}